\documentclass[11pt]{amsart}
\usepackage{amsopn}
\usepackage{amssymb, amscd}
\usepackage{multirow}
\usepackage{graphicx, graphics, epsfig}
\usepackage{faktor} 
\usepackage{enumerate}
\usepackage{tikz}
\usepackage{pgfplots}
\usepackage{hyperref}
\usepackage{color}

\newcommand{\nc}{\newcommand}

\nc{\fg}{\mathfrak{f} } \nc{\vg}{\mathfrak{v} } \nc{\wg}{\mathfrak{w} }
\nc{\zg}{\mathfrak{z} } \nc{\ngo}{\mathfrak{n} } \nc{\kg}{\mathfrak{k} }
\nc{\mg}{\mathfrak{m} } \nc{\bg}{\mathfrak{b} } \nc{\ggo}{\mathfrak{g} } \nc{\eg}{\mathfrak{e} }
\nc{\ggob}{\overline{\mathfrak{g}} } \nc{\sog}{\mathfrak{so} }
\nc{\sug}{\mathfrak{su} } \nc{\spg}{\mathfrak{sp} } \nc{\slg}{\mathfrak{sl} }
\nc{\glg}{\mathfrak{gl} } \nc{\cg}{\mathfrak{c} } \nc{\rg}{\mathfrak{r} }
\nc{\hg}{\mathfrak{h} } \nc{\tg}{\mathfrak{t} } \nc{\ug}{\mathfrak{u} }
\nc{\dg}{\mathfrak{d} } \nc{\ag}{\mathfrak{a} } \nc{\pg}{\mathfrak{p} }
\nc{\sg}{\mathfrak{s} } \nc{\affg}{\mathfrak{aff} } \nc{\qg}{\mathfrak{q} } \nc{\lgo}{\mathfrak{l} }

\nc{\pca}{\mathcal{P}} \nc{\nca}{\mathcal{N}} \nc{\lca}{\mathcal{L}}
\nc{\oca}{\mathcal{O}} \nc{\mca}{\mathcal{M}} \nc{\tca}{\mathcal{T}}
\nc{\aca}{\mathcal{A}} \nc{\cca}{\mathcal{C}} \nc{\gca}{\mathcal{G}}
\nc{\sca}{\mathcal{S}} \nc{\hca}{\mathcal{H}} \nc{\bca}{\mathcal{B}}
\nc{\dca}{\mathcal{D}} \nc{\eca}{\mathcal{E}} \nc{\wca}{\mathcal{W}}

\nc{\vp}{\varphi} \nc{\ddt}{\tfrac{d}{dt}} \nc{\dsdt}{\tfrac{d^2}{dt^2}} \nc{\dds}{\frac{d}{ds}}
\nc{\dpar}{\frac{\partial}{\partial t}} \nc{\im}{\mathrm{i}}

\nc{\SO}{\mathrm{SO}} \nc{\Spe}{\mathrm{Sp}} \nc{\Sl}{\mathrm{SL}}
\nc{\SU}{\mathrm{SU}} \nc{\Or}{\mathrm{O}} \nc{\U}{\mathrm{U}} \nc{\Gl}{\mathrm{GL}}
\nc{\Se}{\mathrm{S}} \nc{\Cl}{\mathrm{Cl}} \nc{\Spin}{\mathrm{Spin}}
\nc{\Pin}{\mathrm{Pin}} \nc{\G}{\mathrm{GL}_n(\RR)} \nc{\g}{\mathfrak{gl}_n(\RR)}

\nc{\RR}{{\Bbb R}} \nc{\HH}{{\Bbb H}} \nc{\CC}{{\Bbb C}} \nc{\ZZ}{{\Bbb Z}}
\nc{\FF}{{\Bbb F}} \nc{\NN}{{\Bbb N}} \nc{\QQ}{{\Bbb Q}} \nc{\PP}{{\Bbb P}} \nc{\OO}{{\Bbb O}}

\nc{\vs}{\vspace{.2cm}} \nc{\vsp}{\vspace{1cm}} \nc{\ip}{\langle\cdot,\cdot\rangle}
\nc{\ipp}{(\cdot,\cdot)} \nc{\la}{\langle} \nc{\ra}{\rangle} \nc{\unm}{\tfrac{1}{2}}
\nc{\unc}{\tfrac{1}{4}} \nc{\und}{\tfrac{1}{16}} \nc{\no}{\vs\noindent}
\nc{\lam}{\Lambda^2(\RR^n)^*\otimes\RR^n} \nc{\tangz}{{\rm T}^{\rm Zar}}
\nc{\nor}{{\sf n}}  \nc{\mum}{/\!\!/} \nc{\kir}{/\!\!/\!\!/}
\nc{\Ri}{\tfrac{4\Ric_{\mu}}{||\mu||^2}} \nc{\ds}{\displaystyle}
\nc{\ben}{\begin{enumerate}} \nc{\een}{\end{enumerate}} \nc{\f}{\frac}
\nc{\lb}{[\cdot,\cdot]} \nc{\isn}{\tfrac{1}{||v||^2}}
\nc{\gkp}{(\ggo=\kg\oplus\pg,\ip)} \nc{\ukh}{(\ug=\kg\oplus\hg,\ip)}
\nc{\tgkp}{(\tilde{\ggo}=\kg\oplus\pg,\ip)}
\nc{\wt}{\widetilde} 
\nc{\iop}{\mathtt{i}} \nc{\jop}{\mathtt{j}}

\nc{\Hess}{\operatorname{Hess}} \nc{\ad}{\operatorname{ad}}
\nc{\Ad}{\operatorname{Ad}} \nc{\rank}{\operatorname{rk}}
\nc{\Irr}{\operatorname{Irr}} \nc{\End}{\operatorname{End}}
\nc{\Aut}{\operatorname{Aut}} \nc{\Inn}{\operatorname{Inn}}
\nc{\Der}{\operatorname{Der}} \nc{\Ker}{\operatorname{Ker}}
\nc{\Iso}{\operatorname{Iso}} \nc{\Diff}{\operatorname{Diff}}
\nc{\Lie}{\operatorname{L}} \nc{\tr}{\operatorname{tr}} \nc{\dif}{\operatorname{d}}
\nc{\sen}{\operatorname{sen}} \nc{\modu}{\operatorname{mod}}
\nc{\CRic}{\operatorname{PP}} \nc{\Cric}{\operatorname{P}} \nc{\Ricci}{\operatorname{Ric}}
\nc{\sym}{\operatorname{sym}} \nc{\herm}{\operatorname{herm}} \nc{\symac}{\operatorname{sym^{ac}}}
\nc{\symc}{\operatorname{sym^{c}}} \nc{\scalar}{\operatorname{Sc}}
\nc{\grad}{\operatorname{grad}} \nc{\ricci}{\operatorname{Rc}} \nc{\kil}{\operatorname{B}} \nc{\cas}{\operatorname{C}} \nc{\lic}{\operatorname{L}}
\nc{\Nor}{\operatorname{Norm}}  \nc{\ricc}{\operatorname{Rc^{c}}}
\nc{\Ricc}{\operatorname{Ric^{c}}} \nc{\ricac}{\operatorname{Rc^{ac}}}
\nc{\Ricac}{\operatorname{Ric^{ac}}} \nc{\Riem}{\operatorname{Rm}} \nc{\Sec}{\operatorname{Sec}}
\nc{\riccig}{\operatorname{ric^{\gamma}}} \nc{\mm}{\operatorname{m}} \nc{\Mm}{\operatorname{M}}
\nc{\Le}{\operatorname{L}} \nc{\tang}{\operatorname{T}}
\nc{\level}{\operatorname{level}} \nc{\rad}{\operatorname{r}}
\nc{\abel}{\operatorname{ab}} \nc{\CH}{\operatorname{CH}} \nc{\Cone}{{\mathcal C}} \nc{\CCone}{\operatorname{CC}} \nc{\CP}{{\mathcal P}}
\nc{\mcc}{\operatorname{mcc}} \nc{\Adj}{\operatorname{Adj}}
\nc{\Order}{\operatorname{O}}  \nc{\inj}{\operatorname{inj}} \nc{\proy}{\operatorname{pr}}
\nc{\vol}{\operatorname{vol}} \nc{\Diag}{\operatorname{Dg}} \nc{\Diagg}{\operatorname{Diag}}
\nc{\Spec}{\operatorname{Spec}} \nc{\Ima}{\operatorname{Im}} \nc{\Rea}{\operatorname{Re}}
\nc{\spann}{\operatorname{span}} \nc{\Aff}{\operatorname{Aff}} \nc{\E}{\operatorname{E}} \nc{\id}{\operatorname{id}} \nc{\dete}{\operatorname{det}} \nc{\Crit}{\operatorname{Crit}} \nc{\val}{\operatorname{val}}

\theoremstyle{plain}
\newtheorem{theorem}{Theorem}[section]
\newtheorem{proposition}[theorem]{Proposition}
\newtheorem{corollary}[theorem]{Corollary}
\newtheorem{lemma}[theorem]{Lemma}

\theoremstyle{definition}
\newtheorem{definition}[theorem]{Definition}

\theoremstyle{remark}
\newtheorem{remark}[theorem]{Remark}

\title{On the stability of homogeneous Einstein manifolds}

\author{Jorge Lauret}  

\address{FaMAF, Universidad Nacional de C\'ordoba and CIEM, CONICET (Argentina)}
\email{jorgelauret@unc.edu.ar} 

\thanks{This research was partially supported by a grant from Univ. Nac. de C\'ordoba, Argentina}

\date{\today}

\begin{document}

\maketitle

\begin{abstract}
Let $g$ be a $G$-invariant Einstein metric on a compact homogeneous space $M=G/K$.  We use a formula for the Lichnerowicz Laplacian of $g$ at $G$-invariant $TT$-tensors to study the stability type of $g$ as a critical point of the scalar curvature function.  The case when $g$ is naturally reductive is studied in special detail.     
\end{abstract}

\tableofcontents

\section{Introduction}\label{intro}

Given a compact connected differentiable manifold $M$ and a transitive action of a compact Lie group $G$ on $M$, the aim of this paper is to study the stability of $G$-invariant Einstein metrics on $M$ within the $G$-invariant setting.  It is well known that if $\mca^G_1$ denotes the finite-dimensional manifold of all unit volume $G$-invariant metrics on $M$, then $g\in\mca_1^G$ is Einstein (i.e.\ $\ricci(g)=\rho g$ for some $\rho\in\RR$, which is necessarily positive if $G$ is non-abelian) if and only if $g$ is a critical point of the scalar curvature functional 
$$
\scalar:\mca_1^G\longrightarrow \RR.  
$$ 
The $G$-action we have fixed provides a presentation $M=G/K$ of $M$ as a homogeneous space, where $K\subset G$ is the isotropy subgroup at some origin point $o\in M$.  

We start by showing in \S\ref{stabhom-sec} that 
$$
T_g\mca_1^G = T_g\Aut(G/K)\cdot g \oplus \tca\tca_g^G, 
$$
where $\Aut(G/K)\subset\Diff(M)$ is the Lie group of automorphisms of $G$ taking $K$ onto $K$, giving rise to trivial variations of $g$, and $\tca\tca_g^G:=(\Ker\delta_g\cap\Ker\tr_g)^G$ is the space of so-called TT-tensors (see \S\ref{stab-sec}) which are $G$-invariant.  It is therefore natural to say that an Einstein metric $g\in\mca_1^G$ is $G$-{\it stable} when the second derivative or Hessian of $\scalar$ satisfies that 
$$
\scalar''_g|_{\tca\tca_g^G}<0, 
$$
which in particular implies that $g$ is a local maximum of $\scalar:\mca_1^G\longrightarrow \RR$.  Recall that without assuming $G$-invariance, $g$ is called {\it stable} if $\scalar''_g$ is negative definite on $\tca\tca_g$, the infinite dimensional space of all unit volume constant scalar curvature (non-trivial) variations of $g$ (see \S\ref{stab-sec}).   

Some potential applications of establishing the $G$-stability type of $G$-invariant Einstein metrics include:  

\begin{enumerate}[{\small $\bullet$}] 
\item If $g$ is $G$-{\it non-degenerate} (i.e., $\scalar''_g|_{\tca\tca_g^G}$ is non-degenerate), then $g$ is $G$-{\it rigid}, in the sense that $g$ is an isolated point in the moduli space $\eca_1^G/\Aut(G/K)$ of $G$-invariant unit volume Einstein metrics on $M$.  The main long standing open question in the subject is whether such moduli space is always finite, which has been conjectured to hold in the multiplicity-free isotropy representation case by B\"ohm, Wang and Ziller in \cite{BhmWngZll} (note that $T_g\mca_1^G =\tca\tca_g^G$ in that case and so $\eca_1^G$ must itself be finite).  

It is worth noticing that since $\eca_1^G$ is known to be compact (see  \cite[Theorem 1.6]{BhmWngZll}), the finiteness of $\eca_1^G/\Aut(G/K)$ is equivalent to the $G$-rigidity of any $G$-invariant Einstein metric on $M$.  $G$-non-degeneracy seems to be a generic property, though this is hard to put in a rigorous statement.  

\item In the case when $g$ is $G$-{\it unstable} (i.e., $\scalar''_g(T,T)>0$ for some $T\in{\tca\tca_g^G}$), one obtains that $g$ is also unstable relative to the $\nu$-entropy functional introduced by Perelman (see \cite{CaoHe}) and so it is {\it dynamically unstable}, in the sense that there exists a nontrivial normalized Ricci flow defined on $(-\infty,0]$ which converges modulo diffeomorphisms to $g$ as $t\to-\infty$ (see \cite[Theorem 1.3]{Krn}).  Additionally, it is known that a $G$-unstable Einstein metric $g$ does not realize the Yamabe invariant of M (see \cite[Theorem 5.1]{BhmWngZll}).

$G$-instability is also an expected behavior, as suggested by the graph theorem \cite[Theorem 3.3]{BhmWngZll} and its generalization, the simplicial complex theorem \cite[Theorem 1.5]{Bhm}.  However, a rigorous result on this is still lacking.  

\item Beyond irreducible symmetric metrics and the special case when $K$ is a maximal subgroup of $G$ (see \cite{WngZll, Bhm}), $G$-stability is extremely rare if $\dim{\mca_1^G}>1$, it is considered a mere coincidence or accident by the experts.  It is for instance unknown whether there can be two non-homothetic $G$-stable Einstein metrics for a given $G$.  

\item Since the normalized Ricci flow on $\mca_1^G$ is precisely the gradient flow of $\scalar$, its dynamical behavior is mostly governed by the $G$-stability types of their fixed points, the $G$-invariant Einstein metrics (see \cite{AnsChr} and references therein).  
\end{enumerate}

As known, the second variation $\scalar''_g$ of the total scalar curvature at any Einstein metric $g$ on $M$, say with $\ricci(g)=\rho g$, coincides on $\tca\tca_g$ with $\unm(2\rho\id-\Delta_L)$, where $\Delta_L$ is the Lichnerowicz Laplacian of $g$ (see \S\ref{stab-sec}).  In \S\ref{mba-sec}, we consider the self-adjoint operator 
$$
\lic_\pg=\lic_\pg(g):\sym(\pg)^K\longrightarrow\sym(\pg)^K, 
$$
defined by $\Delta_L$ under the usual identifications, where $\ggo=\kg\oplus\pg$ is any reductive decomposition and $\sym(\pg)^K:=\{A:\pg\rightarrow\pg:A^t=A, \; [\Ad(K),A]=0\}$.  Note that the $G$-stability type of $g$ is therefore determined by how is the constant $2\rho$ suited relative to the spectrum of $\lic_\pg$.  We use moving bracket approach techniques to prove the following formula for $\lic_\pg$: 
\begin{equation}\label{Lp-intro}
\la \lic_\pg A,A\ra = \unm |\theta(A)\mu_\pg|^2 + 2\tr{\Mm_{\mu_\pg}A^2}, \qquad\forall A\in\sym(\pg), 
\end{equation}
where $\mu_\pg:= \proy_\pg\circ \lb|_{\pg\times\pg} :\pg\times\pg\longrightarrow\pg$ and the function $\Mm:\Lambda^2\pg^*\otimes\pg\rightarrow\sym(\pg)$ is the {\it moment map} from geometric invariant theory (see \cite{alek, BhmLfn}) for the representation $\theta$ of $\glg(\pg)$ given by    
$$
\theta(A)\lambda := A\lambda(\cdot,\cdot) - \lambda(A\cdot,\cdot) - \lambda(\cdot,A\cdot), \qquad \forall A\in\glg(\pg), \quad \lambda\in\Lambda^2\pg^*\otimes\pg,
$$
that is,
$$
\la\Mm_{\mu_\pg},A\ra := \unc\la\theta(A)\mu_\pg,\mu_\pg\ra, \qquad\forall A\in\glg(\pg).   
$$
This is actually the main part of Ricci curvature, the Ricci operator of the metric $g$ is given by $\Ricci(g) = \Mm_{\mu_\pg} - \unm\kil_\mu$, where $\la\kil_\mu\cdot,\cdot\ra:=\kil_{\ggo}|_{\pg\times\pg}$ and $\kil_{\ggo}$ denotes the Killing form of the Lie algebra $\ggo$.   

As a first application of formula \eqref{Lp-intro}, we focus in \S\ref{natred-sec} on the case when $g$ is naturally reductive with respect to $G$ and $\pg$.  We have in this case that 
$$
T_g\mca_1^G =\tca\tca_g^G=\sym_0(\pg)^K:=\{ A\in\sym(\pg)^K:\tr{A}=0\},
$$ 
and furthermore, the operator $\lic_\pg$ is non-negative and takes the following simpler form: 
\begin{equation}\label{Lpnr-intro}
\lic_\pg A:=-\unm\sum[\ad_\pg{X_i},[\ad_\pg{X_i},A]], \qquad\forall A\in\sym(\pg)^K,
\end{equation}
where $\{ X_i\}$ is any $g$-orthonormal basis of $\pg$ and $\ad_\pg{X_i}:=\mu_\pg(X_i,\cdot)$ (recall that naturally reductive means that $\ad_\pg{X_i}$ is skew-symmetric for all $i$).  In particular, if $g_{\kil}$ is the Killing left-invariant metric on any compact simple Lie group $G$, which satisfies $\ricci(g_{\kil})=\unc g_{\kil}$, then 
$$
\lic_\pg(g_{\kil})=\unm \cas_{\tau,-\kil_\ggo},
$$ 
where $\cas_{\tau,-\kil_\ggo}$ is the Casimir operator acting on the representation $\sym(\ggo)$ of $\ggo$ given by $\tau(X)A:=[\ad{X},A]$.  Thus the $G$-stability type of $g_{\kil}$ can be obtained by using representation theory to compute the spectrum of $\cas_{\tau,-\kil_\ggo}$ (see Table \ref{table1}).  We obtain that they are all $G$-stable, except for $\SU(n)$, $n\geq 3$ and $\Spe(n)$, $n\geq 2$, where  $g_{\kil}$ is $G$-neutrally stable of nullity $n^2-1$ and $G$-unstable of coindex $\geq \frac{2n(2n-1)}{2}-1$, respectively.  The picture in the $G$-invariant setting is therefore analogous to the general case, which follows from Koiso's results on the stability of irreducible symmetric spaces (see \S\ref{stab-sec}).  

On the other hand, we use formula \eqref{Lpnr-intro} to compute the matrix of $\lic_\pg$ in the multiplicity-free case in terms of the structural constants of the metric.  Given any $g$-orthogonal decomposition $\pg=\pg_1\oplus\dots\oplus\pg_r$ in $\Ad(K)$-invariant and irreducible subspaces, the numbers
$$
[ijk]:=\sum_{\alpha,\beta,\gamma} g([X_\alpha^i,X_\beta^j], X_\gamma^k)^2,
$$
where $\{ X_\alpha^i\}$ is a $g$-orthonormal basis of $\pg_i$, are invariant under any permutation of $ijk$ by the natural reductivity of $g$ and one has that $\ricci(g)=\rho g$ if and only if 
$$
 \tfrac{b_k}{2} -\tfrac{1}{4d_k}\sum_{i,j} [ijk] = \rho, \qquad\forall k=1,\dots,r,
$$  
where $-\kil_\ggo|_{\pg_k}=b_kg|_{\pg_k}$ and $d_k:=\dim{\pg_k}$.  We obtain in \S\ref{strconst-sec} that the entries of the matrix of $\lic_\pg$ with respect to the orthonormal basis $\left\{ \tfrac{1}{\sqrt{d_1}}I_{\pg_1},\dots, \tfrac{1}{\sqrt{d_r}}I_{\pg_r}\right\}$ of $\sym(\pg)^K$ are given by 
\begin{equation}\label{Lpsc-intro}
[\lic_\pg]_{kk} = 
\tfrac{1}{d_k}\sum_{\substack{j\ne k\\ i}} [ijk], \quad\forall k,
\qquad 
[\lic_\pg]_{jk} =  
-\tfrac{1}{\sqrt{d_j}\sqrt{d_k}}\sum_{i} [ijk], \quad\forall j\ne k.  
\end{equation}

This formula is applied in \S\ref{unif-sec} to prove that the standard metric is $G$-unstable (and consequently Ricci flow dynamically unstable) on each of the following homogeneous spaces,

\begin{enumerate}[{\small $\bullet$}] 
\item $\SU(nk)/\Se(\U(k)\times\dots\times\U(k))$, $\quad k\geq 1$, 

\item $\Spe(nk)/\Spe(k)\times\dots\times\Spe(k)$, $\quad k\geq 1$, 

\item $\SO(nk)/\Se(\Or(k)\times\dots\times\Or(k))$, $\quad k\geq 3$,  
\end{enumerate}
where the quotients are all $n$-times products with $n\geq 3$.  Note that $\dim{\mca^G}=\tfrac{n(n-1)}{2}$.  We also compute the coindex (see Table \ref{table2}) and found that the standard metric is a local minimum of $\scalar:\mca_1^G\rightarrow\RR$ in many cases (including $\SU(3)/T^2$) and it is $G$-degenerate in some others (e.g., $\SU(4)/T^3$).  

As a second application of formula \eqref{Lpsc-intro}, we study in \S\ref{DZ-sec} the $G$-stability of the left-invariant Einstein metrics found by Jensen in \cite{Jns2}.  Given any simple Lie group $H$, one considers the left-invariant metric on $H$ given by  
$$
g_t=-\kil_\hg|_{\ag} + t(-\kil_\hg)|_{\kg}, \qquad t>0,   
$$
where $K\subset H$ is a semisimple subgroup and $\hg = \ag\oplus\kg$ is the $\kil_\hg$-orthogonal decomposition.   $g_1$ is therefore the Killing metric on $H$ and for each $t\ne 1$, the metric $g_t$ is naturally reductive with respect to $G=H\times K$ (see \cite{Zll} or \cite[Theorem 1]{DtrZll}).  If we assume that $\ag$ is $\Ad(K)$-irreducible (i.e., $H/K$ is isotropy irreducible), then the isotropy representation of $G/\Delta K$ is mutliplicity-free and consists of $r+1$ $\Ad(K)$-irreducible summands, where $\kg=\kg_1 \oplus \dots \oplus \kg_r$ is a decomposition in simple ideals of $\kg$.  Note that therefore $\dim{\mca_1^G}=r$.  We also assume that $\kil_{\kg_i}=c\kil_\hg|_{\kg_i}$ for any $i=1,\dots,r$ and some constant $c$.  It is proved in \cite[Corollary 2, p.44]{DtrZll} that $\Ricci(g_t)=\rho I$ ($t\ne 1$) if and only if,
$$
t=t_E:=\tfrac{dc}{(d+2k)(1-c)}, \qquad 2\rho = \tfrac{c}{2t_E}+\tfrac{(1-c)t_E}{2},
$$ 
where $d=\dim{\ag}$ and $k:=\dim{\kg}$.  The explicit computation of $\Spec(\lic_\pg)$ using \eqref{Lpsc-intro} shows that every $g_{t_E}$ is $G$-unstable with coindex $r$, and in particular, $g_{t_E}$ is always a local minimum.  This provides at least one $H$-unstable (and so Ricci flow dynamically unstable) left-invariant Einstein metric on most simple Lie groups, including one of coindex $\geq 3$ on $E_6$ and one of coindex $\geq 2$ on $\SO(2n)$, $\Spe(2n)$, $\SU(n^2)$ and $E_7$.  

Finally, we would like to mention that this is the first of a series of forthcoming papers on $G$-stability of homogeneous Einstein metrics on compact manifolds.  In \cite{stab-dos}, we give a formula for the operator $\lic_\pg(g)$ for any $G$-invariant Einstein metric $g$ in terms of its usual structural constants $[ijk]$ with respect to a bi-invariant metric on $\ggo$.  The formula is used to establish the $G$-stability types of several Einstein metrics on well-known families of homogeneous spaces, including generalized Wallach spaces and some generalized flag manifolds.  On the other hand, we compute in \cite{stab} the $G$-stability types of all the standard Einstein metrics with $G$ simple obtained in the famous classification by Wang and Ziller in \cite{WngZll2}.

\vs \noindent {\it Acknowledgements.}  I am very grateful with Emilio Lauret for computing the first eigenvalue of the Casimir operator in Table \ref{table1}.  I also thank Christoph B\"ohm, Ioannis Chrysikos, McKenzie Wang and Wolfgang Ziller for many helpful conversations.

\section{Stability of compact Einstein manifolds}\label{stab-sec}

Einstein metrics on a compact differentiable manifold $M$, i.e., the Ricci tensor satisfies $\ricci(g)=\rho g$ for some $\rho\in\RR$, were first studied by Hilbert, who proved that they are precisely the critical points of the total scalar curvature functional 
\begin{equation}\label{sct}
\widetilde{\scalar}(g):=\int_M \scalar(g)\; d\vol_g,
\end{equation}
restricted to the space $\mca_1$ of unit volume Riemannian metrics on $M$ (see \cite[4.21]{Bss}).  A fundamental problem is to determine whether a given Einstein metric $g$ is {\it rigid}, in the sense that any Einstein metric sufficiently close to $g$ (compact open $C^\infty$ topology) is isometric to $g$ up to scaling.  Hilbert's variational characterization, beyond being a tool for the existence problem, allows the use of stability theory and calculus of variations in the study of the rigidity of Einstein metrics.  

The case of $(M,g)$ being isometric to a round sphere will be excluded in what follows.  The tangent space $T_g\mca=\sca^2(M)$ (symmetric $2$-tensors) of the space $\mca$ of all Riemannian metrics on $M$ at a metric $g\in\mca$ admits the following decomposition (see \cite[4.57]{Bss}): 
\begin{equation}\label{Tdec}
T_g\mca = \left( \lca_{\mathfrak{X}(M)}g \oplus C^\infty(M)g \right) \oplus^{\perp_g} \tca\tca_g, 
\end{equation}
where $\perp_g$ denotes orthogonality with respect to the usual $L^2$ inner product $\ip_g$ on $\sca^2(M)$ defined by $g$.  The three summands are given by:
\begin{enumerate}[{\small $\bullet$}] 
\item $\lca_{\mathfrak{X}(M)}g=\Ima \delta_g^*=T_g\Diff(M)\cdot g$ is the space of {\it trivial variations}, where $\lca$ denotes Lie derivative.  Here $\delta_g:\sca^2(M)\rightarrow\Omega^1(M)$ is the divergence  operator $\delta_g(T):=-\sum\limits_i \nabla_{X_i}T(X_i,\cdot)$, where $\{ X_i\}$ is any local orthonormal frame, and $ \delta_g^*$ is sometimes called the Killing operator as its kernel consists of Killing vector fields.  An alternative decomposition is given by $T_g\mca=\Ima\delta_g^*\oplus^{\perp_g}\Ker\delta_g$.  

\item $C^\infty(M)g$ is the space of {\it conformal variations}, i.e., the tangent space at $g$ of the space of metrics which are conformally equivalent to $g$.  Note that $\RR g\subset C^\infty(M)g$.  

\item $\tca\tca_g=\Ker\delta_g\cap\Ker\tr_g$ is the subspace of divergence-free (or transversal) and traceless symmetric $2$-tensors, so-called {\it TT-tensors}.     
\end{enumerate}
Let us now assume that $g$ is an Einstein metric on $M$.  If 
$$
\cca:=\{ g\in\mca:\scalar(g)\,\mbox{is a constant function on}\, M\}, 
$$ 
then at any $g\in\cca$, 
\begin{equation}\label{Cdec}
T_g\cca = \left( \lca_{\mathfrak{X}(M)}g \oplus \RR g \right) \oplus^{\perp_g} \tca\tca_g.
\end{equation}
Thus $\tca\tca_g$ can also be described as the space of all unit volume constant scalar curvature non-trivial variations of $g$ (see \cite[4.44-4.46]{Bss}). 

We consider the second variation (or Hessian) of $\widetilde{\scalar}$ at $g$, i.e.,  
$$
\widetilde{\scalar}''_g(T,T) := \left.\dsdt\right|_0 \widetilde{\scalar}(g+tT), \qquad\forall T\in\sca^2(M).     
$$ 
Recall that $g$ is a critical point of $\widetilde{\scalar}|_{\mca_1}$, so for traceless tensors, this can be computed by using, instead of the line $g+tT$, any smooth curve $g(t)\in\mca$ such that $g(0)=g$ and $g'(0)=T$.  The following properties of the second variation are well known (see \cite[4.60]{Bss}): 

\begin{enumerate}[{\small $\bullet$}] 
\item Decomposition \eqref{Tdec} is orthogonal with respect to $\widetilde{\scalar}''_g$, so its restriction on each of the three summands can be studied separately.   

\item $\widetilde{\scalar}''_g$ vanishes on $\lca_{\mathfrak{X}(M)}g$ and $\widetilde{\scalar}''_g(g,g)=2\scalar(g)$.   

\item $\widetilde{\scalar}''_g$ is positive definite on $C^\infty(M)g$.  

\item $\widetilde{\scalar}''_g|_{\tca\tca_g}$ is negative definite on the orthogonal complement of a (possibly trivial) finite-dimensional vector subspace of $\tca\tca_g$ (i.e., nullity and coindex are both finite).    
\end{enumerate}
These facts motivate the definition of the following concepts.  

\begin{definition}\label{stab-def-2}
Let $g\in\mca$ be an Einstein metric.  We call $g$  
\begin{enumerate}[{\small $\bullet$}] 
\item {\it $\scalar$-stable} (or $\scalar$-linearly stable): $\widetilde{\scalar}''_g|_{\tca\tca_g}<0$ (see \cite[Definition 2.7]{Kso} and \cite[Definition 2.2]{CaoHe}).  In particular, $g$ is a local maximum of $\widetilde{\scalar}|_{\cca_1}$ if $g\in\mca_1$, where $\cca_1$ is the space of all unit volume constant scalar metrics on $M$ (indeed, by \eqref{Cdec}, $T_g\cca_1=\lca_{\mathfrak{X}(M)}g \oplus^{\perp_g} \tca\tca_g$ and one uses that $\tca\tca_g$ exponentiates into a slice for the $\Diff(M)$-action; see \cite[12.22]{Bss} or \cite[Lemma 2.6.3]{Krn1}).  This is actually the definition of $g$ $\scalar$-stable in many papers (e.g., \cite{Bhm2,WngWng}).   

\item {\it $\scalar$-unstable} (or $\scalar$-linearly unstable): $\widetilde{\scalar}''_g(T,T)>0$ for some $T\in \tca\tca_g$ (see \cite[Definition 2.7]{Kso} and \cite[Definition 2.2]{CaoHe}).    

\item {\it infinitesimally non-deformable}: $\Ker \E'_g\cap \tca\tca_g=0$, and otherwise {\it infinitesimally deformable} (see \cite[12.29]{Bss}).  Here, $\E'_g$ is the first variation of the operator
\begin{equation}\label{Eop}
\E:\mca\longrightarrow\sca^2(M), \qquad \E(g):=\ricci(g) - \tfrac{\widetilde{\scalar}(g)}{n} g,  
\end{equation}
so-called the {\it Einstein operator} (see \cite[12.26]{Bss}).  Note that $g\in\mca_1$ is Einstein if and only if $\E(g)=0$.  Each element of $\Ker\E'_g\cap \tca\tca_g$ is called an {\it infinitesimally Einstein deformation}, which may or may not be the velocity of a genuine Einstein deformation, i.e., a differentiable curve $g(t)$ of Einstein metrics through $g$.     
\end{enumerate}
\end{definition}

If $\ricci(g)=\rho g$, then for any $T\in\tca\tca_g$, 
$$
\widetilde{\scalar}''_g(T,T) = -\unm\la(\Delta_L-2\rho\id)T,T\ra_g \quad\mbox{and}\quad
\E'_g(T)=\unm\Delta_L(T)-\rho T,
$$ 
where $\Delta_L$ is the {\it Lichnerowicz Laplacian} of $g$, given by,
$$
\Delta_LT=-\nabla^*\nabla T-2\Riem_g(T,\cdot)+\ricci_g\circ T+T\circ\ricci_g,
$$ 
and $\nabla\nabla^*$ denotes the usual rough Laplacian of $g$ (see \cite[4.64]{Bss} and \cite[12.28']{Bss}, respectively).  This implies that if $\lambda_L(g)$ denotes the smallest eigenvalue of $\Delta_L|_{\tca\tca_g}$, then the following characterizations hold (cf.\ \cite[\S 4]{CaoHe} and \cite[\S 1]{WngWng}): 

\begin{enumerate}[{\small $\bullet$}] 
\item $g$ is  $\scalar$-stable if and only if $2\rho<\lambda_L(g)$.  

\item $g$ is $\scalar$-unstable if and only if $\lambda_L(g)< 2\rho$.  

\item $g$ is infinitesimally non-deformable if and only if $2\rho\notin\Spec\left(\Delta_L|_{\tca\tca_g}\right)$, if and only if $\widetilde{\scalar}''_g|_{\tca\tca_g}$ is non-degenerate.  
\end{enumerate}
In particular, stability implies infinitesimal non-deformability (cf.\ \cite[Remark (2) below Definition 2.7]{Kso}).  On the other hand, the fact that any infinitesimally non-deformable Einstein metric is rigid is a strong result by Koiso (see \cite[Proposition 3.3]{Kso} and \cite[12.66]{Bss}).  

After forty years, the stability picture for symmetric spaces has recently been completed.  

\begin{theorem}\cite{Kso, GsqGld, SmmWng, Sch}  
All compact irreducible symmetric spaces are $\scalar$-stable, except for
$$
\begin{array}{c}
\Spe(n)\; (n \geq 2), \quad \Spe(n)/\U(n)\;  (n \geq 3), \quad \SO(5)/(\SO(3) \times \SO(2)), \\ 
\Spe(p+q)/(\Spe(p)\times\Spe(q))\; (p,q\geq 2), 
\end{array}
$$
which are $\scalar$-unstable and infinitesimally non-deformable, and 
$$
\begin{array}{c}
\SU(n)/ \SO(n), \quad \SU(2n)/ \Spe(n)\; (n \geq 3), \\ 
\SU(p + q)/\Se(\U(p) \times \U(q))\; (p \geq q \geq 2), \quad \Spe(3)/(\Spe(2)\times \Spe(1)), \\ 
F_4/ \Spin(9), \quad \SU(n)\; (n \geq 3), \quad E_6/F_4,
\end{array}
$$
which are infinitesimally deformable and not $\scalar$-unstable (i.e., $\lambda_L(g)=2\rho$), often called $\scalar$-neutrally stable.  
\end{theorem}

The following questions remain open: 

\begin{enumerate}[{\small $\bullet$}] 
\item Are the infinitesimally deformable irreducible symmetric metrics local maxima of $\widetilde{\scalar}|_{\cca_1}$?  The only results we know on this question are that $\SU(3)$ and $\SU(2n)/ \Spe(n)$ are not local maxima (see \cite{Jns} and \cite[Example 6.7]{BhmWngZll}, respectively).  We refer to \cite{locmax} for a more detailed treatment of this question.  

\item Does there exist a $\scalar$-stable Einstein manifold with $\scalar>0$ which is not symmetric?  

\item Are the irreducible symmetric spaces 
$$
\SU(n)/\SO(n), \quad \SU(2n)/\Spe(n), \quad \SU(p+q)/\Se(\U(p)\times\U(q)), \quad \SU(n), \quad  E_6/F_4,
$$  
rigid?  Recently, the space $\SU(2n+1)$ has been shown to be rigid in \cite{BttHllMrpWld}.  
\end{enumerate}

Another important kind of stability is $\nu$-entropy stability, relative to the $\nu$-entropy functional $\nu:\mca\longrightarrow\RR$ introduced by Perelman (see \cite{CaoHe} for the definition).  It was proved in \cite{Prl} that $\nu$ is strictly increasing along any Ricci flow solution unless the solution consists of a shrinking gradient Ricci soliton (e.g., an Einstein metric with positive scalar curvature).   

Decomposition \eqref{Tdec} is also $\nu''_g$-orthogonal and $\nu''_g$ also vanishes on $\lca_{\mathfrak{X}(M)}g$ (see \cite{CaoHmlIlm, CaoHe}).

\begin{definition} \cite[Definition 3.3]{CaoHe} An Einstein metric $g\in\mca$ is said to be, 
\begin{enumerate}[{\small $\bullet$}] 
\item {\it $\nu$-stable}: $\nu''_g\leq 0$ (called $\nu$-linearly stable in \cite[Definition 1.2]{WngWng}).  Equivalently, $\nu''_g|_{C^\infty(M)g}\leq 0$ and $\nu''_g|_{\tca\tca_g}\leq 0$.  

\item {\it strictly $\nu$-stable}: $\nu''_g|_{C^\infty(M)g}< 0$ and $\nu''_g|_{\tca\tca_g}< 0$.  

\item {\it neutrally $\nu$-stable}: $g$ is $\nu$-stable and there is a non-zero symmetric $2$-tensor $T$ either in $C^\infty(M)g$ or in $\tca\tca_g$ such that $\nu''_g(T,T)=0$.  

\item {\it $\nu$-unstable}: $\nu''_g(T,T)>0$ for some $T$ either in $C^\infty(M)g$ or $\tca\tca_g$.   
\end{enumerate}
\end{definition}

\begin{remark}
In particular, if $g\in\mca_1$ is strictly $\nu$-stable, then $g$ is a local maximum of $\nu$ among conformal variations of $g$, as well as a local maximum of $\nu|_{\cca_1}$ by \eqref{Cdec} (this is called $\nu$-stable in \cite[Definition 1.2]{WngWng}).  
\end{remark}

Let $\lambda(g)$ denote the first eigenvalue of the Laplacian on functions $\Delta$ of the metric $g$ (i.e., the Laplace-Beltrami operator).   

\begin{theorem}\cite{CaoHmlIlm}
Let $(M,g)$ be a compact Einstein manifold other than the standard sphere, with $\ricci(g)=\rho g$, $\rho>0$.  Then, 

\begin{enumerate}[(i)]
\item $\nu''_g(T,T)>0$ for some $T\in C^\infty(M)g$ if and only if $\lambda(g)<2\rho$ (see also \cite[Lemma 3.5]{CaoHe}).  

\item $\nu''_g(T,T)>0$ for some $T\in \tca\tca_g$ if and only if $\lambda_L(g)<2\rho$ (i.e., $g$ is $\scalar$-unstable). 
\end{enumerate}
\end{theorem}

In particular, 

\begin{enumerate}[{\small $\bullet$}] 
\item $g$ is $\nu$-stable if and only if $2\rho\leq\lambda(g)$ and $2\rho\leq\lambda_L(g)$;  

\item it is neutrally $\nu$-stable if and only if in addition $\lambda(g)= 2\rho$ or $\lambda_L(g)= 2\rho$; 

\item and $g$ is $\nu$-unstable if and only if either $\lambda(g)< 2\rho$ or $\lambda_L(g)< 2\rho$.  
\end{enumerate}

The following notion of stability is more intuitive.  

\begin{definition} \cite[Definition 1.1]{Krn} 
A compact Ricci soliton $(M, g)$ is called {\it dynamically stable} if for any metric $g_0$ near $g$, the normalized Ricci flow starting at $g_0$ exists for all $t\geq 0$ and converges modulo diffeomorphisms to an Einstein metric near $g$, as $t\to\infty$.  On the other hand, (M, g) is said to be {\it dynamically unstable} if there exists a nontrivial normalized Ricci flow defined on $(-\infty,0]$ which converges modulo diffeomorphisms to $g$ as $t\to-\infty$.  
\end{definition}

Kr\"oncke proved that if a compact shrinking Ricci soliton $(M, g)$ is not a local maximizer of $\nu$ (in particular, if $g$ is $\nu$-unstable), then $(M,g)$ is dynamically unstable (see \cite[Corollary 6.2.5]{Krn1} or \cite[Theorem 1.3]{Krn}).  The following implications for a positive scalar curvature Einstein metric follow:
$$
\mbox{$\scalar$-instability} \Rightarrow  \mbox{$\nu$-instability} \Rightarrow  \mbox{dynamical instability}.   
$$

\section{Rigidity and stability of homogeneous Einstein manifolds}\label{stabhom-sec}

In this section, we consider a connected differentiable manifold $M$ (not necessarily compact) and assume that $M$ is homogeneous.  We also fix the transitive action of a Lie group $G$ on $M$, which is assumed to be {\it almost-effective} (i.e., only a discrete subgroup of $G$ acts trivially).  This provides a presentation $M=G/K$ of $M$ as a homogeneous space, where $K\subset G$ is the isotropy subgroup at some origin point $o\in M$.  Neither $G$ nor $K$ are assumed to be connected.  

We denote by $\sca^2(M)^G$ the finite-dimensional vector space of all $G$-invariant symmetric $2$-tensors on $M$, and by $\mca^G\subset\sca^2(M)^G$, the open cone of $G$-invariant Riemannian metrics.  Note that $\mca^G$ is a differentiable manifold with $1\leq\dim\mca^G\leq\frac{n(n+1)}{2}$ and tangent space $T_g\mca^G=\sca^2(M)^G$ at any $g\in\mca^G$, where $n:=\dim{M}$.

\subsection{$G$-rigidity}\label{G-rig-sec}
The Lie group $\Aut(G/K)\subset\Diff(M)$ of all Lie automorphisms of $G$ taking $K$ onto $K$ acts by pullback on $\mca^G$, so each of its orbits consist of pairwise isometric metrics and the orbit $\Aut(G/K)\cdot g$ can be viewed as the trivial $G$-invariant deformations of a metric $g\in\mca^G$.  In this way, $\Aut(G/K)$ acts as the natural `gauge group' in the $G$-invariant setting.  

\begin{remark}
Two $G$-invariant metrics belonging to different $\Aut(G/K)$-orbits may however be isometric via some $\psi\in\Diff(M)$ which is not an automorphism.  This cannot occur for left-invariant metrics on completely solvable Lie groups (see \cite{Alk}).   For $G$ compact, one anyhow has that $T_g\Aut(G/K)\cdot g=T_g(\mca^G\cap\Diff(M)\cdot g)$ for any $g\in\mca^G$ (see Corollary \ref{tan4} below).       
\end{remark}

Rigidity of Einstein metrics among $\mca^G$ can therefore be naturally defined as follows.  

\begin{definition}
An $G$-invariant Einstein metric $g$ is called {\it $G$-rigid} if there exists an open neighborhood $U$ of $g$ in $\mca^G$ such that any Einstein $g'\in U$ belongs to $\Aut(G/K)\cdot g$ up to scaling.   
\end{definition} 

In other words, a $G$-invariant Einstein metric $g$ is $G$-rigid when $g$ is an isolated point in the moduli space $\overline{\eca}^G:=\eca^G/\RR_+\Aut(G/K)$, where 
$$
\eca^G:=\{ g\in\mca^G:g\;\mbox{is Einstein}\},  
$$  
and $\RR_+:=\{ a\in\RR:a>0\}$ acts on $\mca^G$ by scaling.  We note that $\overline{\eca}^G=\eca_1^G/\Aut(G/K)$, where $\eca_1^G:=\eca^G\cap\mca_1^G$ and   
$$
\mca^G_1:=\{ g'\in\mca^G:\dete_{\overline{g}}{g'}=1\}.  
$$ 
Here $\overline{g}$ denotes a fixed background metric in $\mca^G$.  For $G$ compact, $\mca^G_1$ is the space of all $G$-invariant metrics of a given fixed volume.  

The space $\eca^G$ is a real semialgebraic subset (i.e., the set of solutions of finitely many polynomial equalities and inequalities) of $\sca^2(M)^G$ (see \cite[Proposition 1.5]{BhmWngZll}).  The following properties therefore follow from classical theorems of Whitney (see e.g.\ \cite{BchCstRoy}): 

\begin{enumerate}[{\small $\bullet$}] 
\item $\eca^G$ has finitely many connected components.  

\item There is a (local) stratification of $\eca^G$ into real algebraic smooth submanifolds.  

\item Path components and connected components coincide, as $\eca^G$ is locally path-connected.
\end{enumerate}

In the compact case, we have in addition the following major result.  

\begin{theorem}\cite[Theorem 1.6]{BhmWngZll}
Let $G$ be a compact Lie group and $M=G/K$ be a connected homogeneous space with finite fundamental group. Then each connected component of $\eca_1^G$ is compact, and the set of possible Einstein constants of metrics among $\eca_1^G$ is finite.
\end{theorem}

In particular, in the compact case, the moduli space $\overline{\eca}^G=\eca_1^G/\Aut(G/K)$ is also compact and hence $\overline{\eca}^G$ is finite if and only if every $g\in\mca^G$ is $G$-rigid.  It is an open question whether $\overline{\eca}^G$ is always finite.  This has been conjectured for the multiplicity-free isotropy representation case in \cite{BhmWngZll}, where only finitely many trivial deformations are possible, so conjecturally, $\eca_1^G$ is itself a finite set.  Classes of compact homogeneous spaces for which $\overline{\eca}^G$ is known to be finite include D'Atri-Ziller metrics (see \cite{DtrZll}), generalized Wallach spaces (see \cite{LmsNknFrs}) and spaces with only two isotropy summands (see \cite{WngZll}), but it is still open in general for generalized flag manifolds, even for the full flag $\SU(n)/T$ for $n$ large.    

On the other hand, a left-invariant Einstein metric on a solvable Lie group $G$ is known to be $G$-rigid; moreover, $\overline{\eca}^G$ is either empty or a singleton (see \cite{Hbr} and \cite[Corollary 4.3]{BhmLfn}).  

\begin{proposition}\label{G-rig}
If an Einstein metric $g\in\mca^G$ is not $G$-rigid, then there exists a smooth path $g:(-\epsilon,\epsilon)\rightarrow\mca^G$ such that $g(0)=g$, $g(s)$ is Einstein for all $s$ and 
$$
g'(0)\perp_g T_g \RR_+\Aut(G/K)\cdot g. 
$$  
\end{proposition}

\begin{remark}
It follows from the existence of a slice for the $\RR_+\Aut(G/K)$-action that the path $g(s)$ is transversal to $\RR_+\Aut(G/K)$-orbits for sufficiently small $\epsilon$, in the sense that $g(s)\notin\RR_+\Aut(G/K)\cdot g(s')$ for all $s,s'\in(-\epsilon,\epsilon)$, $s\ne s'$.  In other words, $g(s)$ descends to a genuine curve through the class of $g$ in the moduli space $\overline{\eca}^G$.   
\end{remark}

\begin{proof}
As an element of $\eca^G$, the metric $g$ belongs to a finite number of connected smooth submanifolds contained in $\eca^G$, each of which is invariant under the connected component $\Aut(G/K)^0$ of the Lie group $\Aut(G/K)$.  Since $g$ is not $G$-rigid, the dimension of the orbit $\RR_+\Aut(G/K)^0\cdot g$ is strictly less than the dimension of at least one of these submanifolds, so the existence of the smooth path $g(s)$ follows.  
\end{proof}

\subsection{Variational principle}\label{G-var-sec}
The manifold $\mca^G$ is itself naturally endowed with a Riemannian metric defined at each $g\in\mca^G$ by
\begin{equation}\label{metg}
\la T,T\ra_g := \tr{A^2}, \quad \mbox{where} \quad T_o=g_o(A\cdot,\cdot), \quad\forall T\in\sca^2(M)^G.
\end{equation}
Note that the linear map $A:T_oM\rightarrow T_oM$ is $g_o$-self-adjoint and $\tr_g{T}=\tr{A}$, $\det_g{T}=\det{A}$.  Equivalently, $\la T,T\ra_g:=\sum T_o(X_i,X_i)^2$, for any $g_o$-orthonormal basis $\{ X_i\}$ of $T_oM$.  In particular, $\ip_g$ is precisely the $L^2$ metric considered in \S \ref{stab-sec} if $M$ is compact and $g\in\mca_1^G$.   

In the case when $G$ is unimodular, it is well known (see e.g.\ \cite{Nkn, Hbr} and \cite[(1.11)]{Wng}) that relative to such metric on $\mca^G$, the gradient of the scalar curvature function 
$$
\scalar:\mca^G\rightarrow\RR, \qquad \scalar(g):=\tr_g{\ricci(g)}, 
$$ 
is given by 
\begin{equation}\label{grad-sc}
\grad(\scalar)_g=-\ricci(g), \qquad\forall g\in\mca^G, 
\end{equation}
where $\ricci(g)\in\sca^2(M)^G$ is the Ricci tensor of $g$.  Since the tangent space of the submanifold 
$\mca^G_1$ at a metric $g\in\mca^G_1$ is precisely 
$$
(\RR g)^{\perp_g} = \left\{ T\in\sca^2(M)^G:\tr_g{T}=0\right\} = \Ker\tr_g\cap\sca^2(M)^G,
$$
one obtains the following result, which it was first proved by Palais (see \cite[4.23]{Bss}) for $G$ compact.  

\begin{lemma}\label{Palais}
If $M=G/K$ and $G$ is unimodular, then $g\in\mca^G_1$ is a critical point of $\scalar|_{\mca^G_1}$ if and only if $g$ is Einstein.  
\end{lemma} 

This variational characterization has been successfully applied for decades, since the pioneer articles \cite{Jns, WngZll}, to study the existence of invariant Einstein metrics on homogeneous spaces (see \cite{BhmWngZll,Bhm,Wng} and references therein).  In this paper, we aim to use the second variation of $\scalar:\mca^G\rightarrow\RR$ to study $G$-rigidity.

\subsection{Trivial variations}\label{G-trivial}
According to \S \ref{G-rig-sec}, the space of trivial $G$-invariant variations of a metric $g\in\mca^G$ is given by the tangent space $T_g\Aut(G/K)\cdot g\subset\sca^2(M)^G$.  A distinguished subgroup of $\Aut(G/K)$ is the normalizer $N_G(K)$, which acts on $M$ by $n\cdot(a\cdot o)=I_n(a\cdot o):=nan^{-1}\cdot o$ and on $T_oM\equiv \ggo/\kg$ by $n\cdot X:=\Ad(n)X$.  Alternatively, the Lie group $N:=N_G(K)/K$ acts on $M$ by {\it $G$-equivariant diffeomorphisms} (i.e., $\psi(a\cdot p)=a\cdot \psi(p)$ for all $a\in G$, $p\in M$) in the following way: $n\cdot (a\cdot o)=R_n(a\cdot o):=an\cdot o$.  Thus $N\cdot g$ is contained in the so-called {\it $G$-equivariant isometry class} of the metric $g$, and since $R_n^*g=I_{n^{-1}}^*g$ for any $n\in N$, one obtains that 
\begin{equation}\label{norm}
N\cdot g=N_G(K)\cdot g, \qquad \forall g\in\mca^G.    
\end{equation}

We consider any {\it reductive decomposition} $\ggo=\kg\oplus\pg$ of the homogeneous space $M=G/K$ (i.e., $\Ad(K)\pg\subset\pg$), where $\ggo$ and $\kg$ are respectively the Lie algebras of $G$ and $K$, which provides the usual identification $T_oM\equiv\pg$.  Thus $\sca^2(M)^G$ will be often identified, without any further mention, with the vector space of $\Ad(K)$-invariant symmetric $2$-forms on $\pg$, and $\mca^G$ with the open cone of positive definite ones.  For each $X\in\pg$, consider the linear map 
\begin{equation}\label{adp}
\ad_\pg{X}:=\proy_\pg\circ \ad{X}|_\pg:\pg\rightarrow\pg,
\end{equation} 
where $\proy_\pg:\ggo\rightarrow\pg$ is the projection on $\pg$ relative to $\ggo=\kg\oplus\pg$. 

As shown in \cite[Lemma 6.10]{PRP}, at any $g\in\mca^G$, the trivial variations space satisfies that
\begin{equation}\label{tg}
T_g\Aut(G/K)\cdot g \subset \left\{ g_o(S(D)\cdot,\cdot):\underline{D}\in\Der(\ggo/\kg)\right\},
\end{equation}
where $S(A):=\unm(A+A^t)$ denotes the symmetric part of a linear map $A$ with respect to $g_o$ and  
$$
\Der(\ggo/\kg):=\{ \underline{D}\in\Der(\ggo):\underline{D}(\kg)\subset\kg\}, \qquad \underline{D}=\left[\begin{matrix} \ast&\ast\\ 0&D\end{matrix}\right].  
$$ 
We note that if 
$$
\pg_0:=\{ X\in\pg:[\kg,X]=0\}, 
$$ 
then $\ad{\pg_0}\subset\Der(\ggo/\kg)$ and the Lie algebra of $N_G(K)$ is given by $N_\ggo(\kg)=\kg\oplus\pg_0$.  On the other hand, $g_o(S(\ad_\pg{\pg})\cdot,\cdot)\cap\sca^2(M)^G\subset g_o(S(\ad_\pg{\pg})^\kg\cdot,\cdot)$, where 
$$
S(\ad_\pg{\pg})^\kg:=\left\{S(\ad_\pg{X}):X\in\pg, \,[\ad{\kg}|_\pg,S(\ad_\pg{X})]=0\right\},  
$$
and equality holds if $K$ is connected.  

\begin{lemma}\label{tan3}
For any $g\in\mca^G$, $S(\ad_\pg{\pg})^\kg = S(\ad_\pg{\pg_0})$ and 
$$
T_gN\cdot g = g_o(S(\ad_\pg{\pg_0})\cdot,\cdot).
$$
\end{lemma}

\begin{remark}\label{tan3-rem}
In the Lie group case, i.e., $M=G$ and $K$ trivial, we have that $S(\ad_\pg{\pg_0})=S(\ad{\ggo})$, so it is zero if and only if $g$ is bi-invariant.  
\end{remark}

\begin{proof}
Since $[\ad{Z}|_\pg,S(\ad_\pg{X})] = S(\ad_\pg{[Z,X]}|_\pg)$ for any $Z\in\kg$, we obtain that $S(\ad_\pg{\pg_0})\subset S(\ad_\pg{\pg})^\kg$.  Conversely, given $S(\ad_\pg{X})\in S(\ad_\pg{\pg})^\kg$, we consider the decomposition $X=X_0+X_1$, where $X_0\in\overline{\pg}:=\{ Y\in\pg:(\ad_\pg{Y})^t=-\ad_\pg{Y}\}$ and $X_1\perp\overline{\pg}$.  Note that both $\overline{\pg}$ and its orthogonal complement are $\ad{\kg}|_\pg$-invariant subspaces.  Thus $[Z,X]$ and $[Z,X_0]$ both belong to $\pg_0$ and so $[Z,X_1]=0$ for any $Z\in\kg$, from which follows that  $S(\ad_\pg{X})=S(\ad_\pg{X_1})\in S(\ad_\pg{\pg_0})$.   

The second equality can be proved using \eqref{norm} as follows.  For any $X\in\ggo$ such that $[X,\kg]\subset\kg$, 
\begin{align*}
\left.\ddt\right|_0 (I_{\exp{tX}})^*g =& \left.\ddt\right|_0 g_o(\Ad(\exp{tX})|_\pg\cdot,\Ad(\exp{tX})|_\pg\cdot) \\ 
=& g_o(\ad{X}|_\pg\cdot,\cdot)+g_o(\cdot,\ad{X}|_\pg\cdot).  
= 2g_o(S(\ad{X}|_\pg)\cdot,\cdot),
\end{align*}
Now if $X=X_\kg+X_\pg$, then $S(\ad{X}|_\pg)=S(\ad_\pg{X_\pg})$ (since $\ad{X_\kg}|_\pg$ is skew-symmetric) and $[X_\pg,\kg]\subset\kg\cap\pg=0$, i.e., $X_\pg\in\pg_0$.    
\end{proof}

Assume from now on in this subsection that $G$ is compact, thus $M$ and $K$ are also compact.  In this case, it is known that $N$ is the group of all $G$-equivariant diffeomorphisms of $M=G/K$ (see \cite[Chapter I, Corollary 4.3]{Brd}) and so $N$-orbits (or $N_G(K)$-orbits, see \eqref{norm}) are precisely the equivariant isometry classes.  Since $N_G(K)$ and $\Aut(G/K)$ have the same connected components of the identity, an Einstein metric $g$ is $G$-rigid if and only if any other $G$-invariant Einstein metric on $M$ near $g$ is equivariantly isometric up to scaling to $g$.  Furthermore, one obtains from Lemma \ref{tan3} the following useful description of the space of trivial $G$-invariant variations.

\begin{corollary}\label{tan2}
If $G$ is compact, then at any $g\in\mca^G$, 
$$
T_g\Aut(G/K)\cdot g = T_gN\cdot g = g_o(S(\ad_\pg{\pg_0})\cdot,\cdot).  
$$  
\end{corollary}

Contrary to what happens in the Lie group case (see Remark \ref{tan3-rem}), the space of trivial variations vanishes in many cases if $K$ is non-trivial:

\begin{enumerate}[{\small $\bullet$}] 
\item If $g\in\mca^G$ is {\it naturally reductive} with respect to $G$, i.e., there exists a reductive decomposition $\ggo=\kg\oplus\pg$ such that $\ad_\pg{X}$ is skew-symmetric for any $X\in\pg$, then $T_g\Aut(G/K)\cdot g=0$ by Corollary \ref{tan2}. 

\item Another direct consequence of Corollary \ref{tan2} is that $T_g\Aut(G/K)\cdot g=0$ for any $g\in\mca^G$ if the trivial representation does not appear in the $\kg$-isotropy representation of $M=G/K$ (i.e., $\pg_0=0$).  

\item If $G$ is compact and the isotropy representation of $G/K$ is {\it mutiplicity-free} (i.e., any two different $\Ad(K)$-invariant irreducible subspaces are inequivalent as $\Ad(K)$-representations), e.g., when $\rank(G)=\rank(K)$, then $N_G(K)\cdot g$ is finite and so $T_g\Aut(G/K)\cdot g=0$ for any $g\in\mca^G$.   
Indeed, the multiplicity-free condition is equivalent to the existence of only finitely many $\Ad(K)$-invariant subspaces of $\pg$, which implies that the connected component $N_G(K)^0$ necessarily leaves invariant any $\Ad(K)$-invariant and irreducible subspace of $\pg$ and consequently $N_G(K)^0$ acts trivially on $\mca^G$.  
\end{enumerate}

\subsection{$G$-invariant TT-tensors}\label{G-TT-sec} 
Recall from \S \ref{stab-sec} the divergence operator $\delta_g$ attached to a Riemannian metric $g$, and the space of TT-tensors $\tca\tca_g=\Ker\delta_g\cap\Ker\tr_g$.  The proof of the following lemma is strongly based on the proof of \cite[Lemma 2.2]{WngWng}.   

\begin{proposition}\label{G-TT}
If $G$ is unimodular and $g\in\mca^G$, then 
$$
\sca^2(M)^G = g_o(S(\ad_\pg{\pg_0})\cdot,\cdot) \oplus^{\perp_g} \Ker\delta_g\cap\sca^2(M)^G.  
$$
\end{proposition}

\begin{remark}\label{G-TT-rem}
In particular, $\sca^2(M)^G\subset\Ker\delta_g$ and so $\tca\tca_g^G= \sca^2(M)^G\cap \Ker\tr_g$ under any of the above three assumptions, where 
$$
\tca\tca_g^G:= \sca^2(M)^G\cap \tca\tca_g
$$ 
is the space of all $G$-invariant TT-tensors.        
\end{remark}

\begin{proof}
Let $\{ X_i\}$ be a $g_o$-orthonormal basis of $\pg$ and extend it to a local frame of Killing vector fields.  Consider $T\in\sca^2(M)^G$. Then, at the point $o$ we have that
\begin{align*}
\delta_g(T)(X) =& -\sum (\nabla_{X_i}T)(X_i,X) = \sum -X_i(T(X_i,X)) + T(\nabla_{X_i}X_i,X) + T(X_i,\nabla_{X_i}X) \\ 
=& \sum T(X_i,[X,X_i]) + T(X_i,\nabla_{X_i}X) + T(\nabla_{X_i}X_i,X) \\ 
=& \sum T(X_i,\nabla_{X}X_i) + T(\nabla_{X_i}X_i,X) \\
=& \sum g(\nabla_{X}X_i,X_k)T(X_i,X_k) + \sum g(\nabla_{X_i}X_i,X_k)T(X_k,X). 
\end{align*}
It follows from the Koszul formula (recall that $[X_i,X_j]_o=-[X_i,X_j]_\pg$, where $\lb_\pg$ denotes the Lie bracket of $\ggo$ restricted and then projected on $\pg$) that the right summand equals 
\begin{align*}
\sum g_o([X_k,X_i]_\pg,X_i)T(X_k,X) =& \sum_k T(X_k,X) \sum_i g_o([X_k,X_i]_\pg,X_i) \\ 
=& \sum T(X_k,X)\tr{\ad_\pg{X_k}} =0,
\end{align*}
since $\tr{\ad_\pg{Y}}=\tr{\ad{Y}}=0$ for any $Y\in\pg$ as $G$ is unimodular, and the left one gives
\begin{align*}
&-\unm\sum g_o([X,X_i]_\pg,X_k)T(X_i,X_k) -\unm\sum g_o([X_i,X_k]_\pg,X)T(X_i,X_k) \\ 
&+\unm\sum g_o([X_k,X]_\pg,X_i)T(X_i,X_k) = -\unm\sum T([X,X_i]_\pg,X_i) -\unm\sum T([X,X_k]_\pg,X_k) \\ 
=& -\sum T([X,X_i]_\pg,X_i) = -\la T,g_o(S(\ad_\pg{X})\cdot,\cdot)\ra_g.  
\end{align*}
Note that the middle term vanishes since $\lb_\pg$ and $T$ are respectively skew-symmetric and symmetric bilinear forms.  Thus a tensor $T\in\sca^2(M)^G$ is divergence-free if and only if $T\perp g_o(S(\ad_\pg{X})\cdot,\cdot)$ for any $X\in\pg$, which is equivalent to $T\perp g_o(S(\ad_\pg{\pg_0})\cdot,\cdot)$ by Lemma \ref{tan3} and the fact that $T$ is $\Ad(K)$-invariant.   
\end{proof}

It follows from Corollary \ref{tan2} and Proposition \ref{G-TT} that the space of all $G$-invariant variations $T_g\mca^G=\sca^2(M)^G$ admits the following decomposition in the compact case.  

\begin{corollary}\label{tan4}
If $G$ is compact, then at any $g\in\mca^G$, 
$$
T_g\mca^G = \RR g\oplus^{\perp_g} T_g\Aut(G/K)\cdot g \oplus^{\perp_g} \tca\tca_g^G.   
$$
\end{corollary}

Recall from Remark \ref{G-TT-rem} that $T_g\mca^G = \RR g\oplus^{\perp_g} \tca\tca_g^G$ therefore holds in many natural cases.  Curiously enough, as far as we know, $S^2\times S^3=\SO(4)/\SO(2)$ is the only homogeneous space $G/K$ with $\dim{K}>0$ known such that $T_g\Aut(G/K)\cdot g$ is nonzero for a $G$-invariant Einstein metric $g$ (see \cite[Example 3.7]{stab-dos}).

\subsection{$G$-stability}\label{G-stab-sec} 
Since the function $\scalar$ is constant on $\Aut(G/K)\cdot g$, its second variation $\scalar''_g$ vanishes on $T_g\Aut(G/K)\cdot g$.  Note that $\scalar''(g,g)=2\scalar(g)$.  On the other hand, if $g\in\mca^G$ is Einstein, then the orbit $\Aut(G/K)\cdot g$ consists of Einstein metrics and so $\E|_{\Aut(G/K)\cdot g} \equiv 0$ and $\E(\RR_+g)=0$, where 
$$
\E:\mca^G\longrightarrow\sca^2(M)^G, \qquad \E(g'):=\ricci(g')-\tfrac{\scalar(g')}{n}g',
$$  
is the Einstein operator or traceless Ricci tensor (cf.\ \eqref{Eop}).  

At each $g\in\mca^G$, we consider the following decomposition,
\begin{equation}\label{tg2}
T_g\mca^G = \left(\RR g \oplus T_g\Aut(G/K)\cdot g\right) \oplus^{\perp_g} W_g,  
\end{equation} 
where $W_g$ is defined as the $\ip_g$-orthogonal complement of the space $\RR g \oplus T_g\Aut(G/K)\cdot g$ of trivial variations.  According to Proposition \ref{G-TT} and \eqref{tg2}, if $G$ is unimodular, then $W_g\subset\tca\tca_g^G$, and if in addition $G$ is compact, then by Corollary \ref{tan4}, 
\begin{equation}\label{WTT} 
W_g=\tca\tca_g^G,
\end{equation}
the vector space of $G$-invariant TT-tensors.  

\begin{remark}
The existence of $G$-invariant Einstein metrics on $M=G/K$ for a non-compact unimodular $G$ is open.  It is proved in \cite{DttLtMtl} that $G$ must be semisimple, hence such existence would provide a counterexample to the {\it Alekseevsky conjecture}: any non-compact and non-flat homogeneous Einstein manifold is isometric to a simply connected solvmanifold (in particular, diffeomorphic to the Euclidean space).  After the conclusion of the first version of this paper, a proof of the Alekseevsky conjecture was uploaded to arXiv by C. B\"ohm and R. Lafuente (see \cite{BhmLfn2}).   
\end{remark}

We are now ready to define the notions of stability and deformability in the $G$-invariant setting (cf.\ Definition \ref{stab-def-2}).   

\begin{definition}\label{stab-def}
An Einstein metric $g\in\mca^G_1$ is said to be,  
\begin{enumerate}[{\small $\bullet$}] 
\item {\it $G$-stable}: $\scalar''_g|_{W_g\times W_g}<0$ (in particular, $g$ is a local maximum of $\scalar|_{\mca^G_1}$, by using a slice for the $\Aut(G/K)$-action on $\mca^G$). 

\item {\it $G$-unstable}: $\scalar''_g(T,T)>0$ for some $T\in W_g$ ($g$ is a saddle point, unless $\scalar''_g|_{W_g\times W_g}>0$, see below).  The {\it coindex} is the dimension of the maximal subspace of $W_g$ on which $\scalar''_g$ is positive definite.   

\item {\it $G$-non-degenerate}: $\scalar''_g|_{W_g\times W_g}$ non-degenerate (thus $g$ is an isolated critical point up to the $\Aut(G/K)$-action, i.e., $g$ is rigid), and otherwise, {\it $G$-degenerate}.  The {\it nullity} is the dimension of the kernel of $\scalar''_g|_{W_g\times W_g}$.  Recall from \S\ref{stab-sec} that $G$-non-degeneracy is equivalent to {\it $G$-infinitesimal non-deformability}: $\Ker d\E|_g\cap W_g=0$, 
where $d\E|_g:\sca^2(M)^G\rightarrow\sca^2(M)^G$ is the derivative of $\E$.  

\item {\it $G$-neutrally stable}: $\scalar''_g|_{W_g\times W_g}\leq 0$ and degenerate (i.e., $g$ is $G$-degenerate and it is not $G$-unstable).  Note that this must hold for any local maximum.  

\item {\it $G$-strongly unstable}: $\scalar''_g|_{W_g\times W_g}> 0$ ($g$ is therefore a local minimum of $\scalar|_{\mca^G_1}$).  
\end{enumerate}
\end{definition}

\begin{remark}
Recall that the prefix $G$ in the name of the different notions is referring not only to the group $G$ but also to its action on $M$, which has been fixed at the beginning of the section.  
\end{remark}

If an Einstein metric $g\in\mca^G$ is $G$-stable, then $g$ is clearly $G$-non-degenerate, which in turn implies that $g$ is $G$-rigid by Proposition \ref{G-rig}.  On the other hand, it follows from \eqref{WTT} and \S\ref{stab-sec} that if $G$ is compact, then
$$
\mbox{$G$-instability} \Rightarrow  \mbox{$\scalar$-instability} \Rightarrow  \mbox{$\nu$-instability} \Rightarrow  \mbox{dynamical instability},  
$$
and that non-rigidity also follows from the assumption that the corresponding $G$-invariant concept holds.      

In \cite[Theorems 1.3, 1.4, 1.5]{WngWng}, the authors obtained that all Einstein metrics on Aloff-Wallach spaces are $G$-unstable, as well as any $G$-invariant Einstein metric on a homogeneous space $G/K$ ($(G,K)$ not a symmetric pair) of dimension $\leq 7$, except for $\SU(2)\times\SU(2)$ and the isotropy irreducible $\Spe(2)/\SU(2)$ (see also \cite{SmmWngWng}).   

\begin{remark}
In \cite{PRP}, the Ricci curvature function 
$$
\ricci:\mca^G\rightarrow\sca^2(M)^G, \qquad g\mapsto\ricci(g),
$$
and its derivative $d\ricci|_g:\sca^2(M)^G\rightarrow\sca^2(M)^G$, at each $g\in\mca^G$, were used in the study of the prescribed Ricci curvature problem.  Given an Einstein metric $g\in\mca^G$, say $\ricci(g)=\rho g$, it is easy to see that restricted to $(\RR g)^{\perp_g}$, $d\E|_g = d\ricci|_g - \rho\id$.  On the other hand, we will show below in \S \ref{mba-sec} that $\scalar''_g(T,T) = \la (\rho\id-d\ricci|_g)T,T\ra_g$, for any $T\in \sca^2(M)^G$.  Thus the stability type of $g$ is determined by $\Spec\left(d\ricci|_g|_{W_g}\right)$.  The operator $d\ricci|_g$, which restricted to $W_g$ is precisely one half of the Lichnerowicz Laplacian $\Delta_L$ when $G$ is compact, was computed in \cite{PRP} in terms of the moment map of the variety of algebras via the moving bracket approach.  This is developed in \S \ref{mba-sec}.  
\end{remark}

\section{Second variation of the scalar curvature}\label{mba-sec}

Given $M^n=G/K$ as in \S 3, we consider any reductive decomposition $\ggo=\kg\oplus\pg$ in order to obtain the usual identifications $T_oM\equiv\pg$ and 
$$
\sca^2(M)^G\leftrightarrow\sym^2(\pg)^K, \qquad \mca^G\leftrightarrow\sym_+^2(\pg)^K,
$$
where $\sym^2(\pg)^K$ is the vector space of all $\Ad(K)$-invariant symmetric $2$-forms on the $n$-dimensional vector space $\pg$ and $\sym_+^2(\pg)^K$ the open cone of positive ones. 

\begin{remark}
It is usual in the literature the choice of $\pg$ as the orthogonal complement of $\kg$ relative to some bi-invariant inner product on $\ggo$, which always exists for $G$ compact.  However, this choice may hide, among other nice properties, the fact that a metric is naturally reductive with respect to $G$.  
\end{remark}

We also fix a background metric $g\in\mca^G$ and set $\ip:=g_o\in\sym_+^2(\pg)^K$.  This allows the following alternative identifications in terms of operators:
$$
\sym(\pg)^K \ni A\leftrightarrow T=\la A\cdot,\cdot\ra\in\sym^2(\pg)^K, \qquad
\sym_+(\pg)^K \ni h\leftrightarrow \la h\cdot,h\cdot\ra\in\sym_+^2(\pg)^K,
$$
where $\sym(\pg)$ is the vector space of all self-adjoint (or symmetric) linear maps of $\pg$ with respect to $\ip$ and $\sym_+(\pg)$ the open subset of those which are positive definite.  Note that $A\in\sym(\pg)$ belongs to $\sym(\pg)^K$ if and only if $[\Ad(K)|_\pg,A]=0$ (equivalently, $[\ad{\kg}|_\pg,A]=0$, if $K$ is connected).

\subsection{Ricci curvature} 
Let $\mu$ denote the Lie bracket of $\ggo$.  We extend $\ip$ in the usual way to inner products on $\glg(\pg)$ and $\Lambda^2\pg^*\otimes\pg$, respectively:
$$
\la A,B\ra:= \tr{AB^t}, \qquad \la\lambda,\lambda\ra:=\sum |\ad_\lambda{X_i}|^2 = \sum |\lambda(X_i,X_j)|^2,
$$
where $\{ X_i\}$ is any orthonormal basis of $\pg$ relative to $\ip$.  We also consider the algebra product, 
\begin{equation}\label{mup}
\mu_\pg:= \proy_\pg\circ \mu|_{\pg\times\pg} :\pg\times\pg\longrightarrow\pg,
\end{equation}
where $\proy_\pg:\ggo\rightarrow\pg$ is the projection on $\pg$ relative to $\ggo=\kg\oplus\pg$, and consider the linear maps $\ad_\pg{X}:=\mu_\pg(X,\cdot)$, $X\in\pg$, as in \eqref{adp}.

If $G$ is unimodular, then the Ricci operator $\Ricci(g)$ of the metric $g$ (see e.g.\ \cite[(5)]{PRP}) is given by
\begin{equation}\label{Ric2}
\Ricci(g) = \Mm_{\mu_\pg} - \unm\kil_\mu,
\end{equation}
where $\la\kil_\mu\cdot,\cdot\ra:=\kil_{\ggo}|_{\pg\times\pg}$, $\kil_{\ggo}$ denotes the Killing form of the Lie algebra $\ggo$ and
\begin{equation}\label{mm2}
\la\Mm_{\mu_\pg},A\ra := \unc\la\theta(A)\mu_\pg,\mu_\pg\ra, \qquad\forall A\in\glg(\pg).  
\end{equation}
Here $\theta$ is the representation of $\glg(\pg)$ given by,
\begin{equation}\label{tita}
\theta(A)\lambda := A\lambda(\cdot,\cdot) - \lambda(A\cdot,\cdot) - \lambda(\cdot,A\cdot), \qquad \forall A\in\glg(\pg), \quad \lambda\in\Lambda^2\pg^*\otimes\pg.
\end{equation}
The function $\Mm:\Lambda^2\pg^*\otimes\pg\rightarrow\sym(\pg)$ is therefore the {\it moment map} from geometric invariant theory (see e.g.\ \cite{BhmLfn} and the references therein) for the representation $\theta$ of $\glg(\pg)$.  Equivalently, 
\begin{equation}\label{mm3}
\Mm_{\mu_\pg} = -\unm\sum (\ad_\pg{X_i})^t\ad_\pg{X_i} + \unc\sum \ad_\pg{X_i}(\ad_\pg{X_i})^t,
\end{equation}
or
\begin{equation}\label{mm4}
\la\Mm_{\mu_\pg}X,X\ra = -\unm\sum \la\mu_\pg(X,X_i),X_j\ra^2+ \unc\sum \la\mu_\pg(X_i,X_j),X\ra^2, \qquad\forall X\in\pg.  
\end{equation} 
It is easy to check that both operators $\Mm_{\mu_\pg}$ and $\kil_\mu$ belong to $\sym(\pg)^K$.  The main part of the Ricci curvature of $g$ is $\Mm_{\mu_\pg}$, observe that $\kil_\mu$ is just measuring in some sense how far is $g$ from being standard.  It follows from \eqref{mm2} and \eqref{tita} that $\tr{\Mm_{\mu_\pg}}=\la\Mm_{\mu_\pg},I\ra=-\unc|\mu_\pg|^2$ and so by \eqref{Ric2},
\begin{equation}\label{scal2}
\scalar(g) = -\unc|\mu_\pg|^2  - \unm\tr{\kil_\mu}.
\end{equation}
We refer to \cite{alek} for more details on this viewpoint on Ricci curvature.

\subsection{Moving bracket approach} 
Recall that $\mu$ is the Lie bracket of $\ggo$.  Given $h\in\sym_+(\pg)^K$, we consider the new Lie algebra $(\ggo,\underline{h}\cdot\mu)$, where $\underline{h}\in\Gl(\ggo)$ is defined by $\underline{h}|_\kg:=I$, $\underline{h}|_\pg:=h$.  Here $\underline{h}\cdot\mu:=\underline{h}\mu(\underline{h}^{-1}\cdot,\underline{h}^{-1}\cdot)$  is the usual action of $\Gl(\ggo)$ on $\Lambda^2\ggo^*\otimes\ggo$, so $\underline{h}:(\ggo,\mu)\rightarrow(\ggo,\underline{h}\cdot \mu)$ is a Lie algebra isomorphism.  Now for  any Lie group $G_{\underline{h}\cdot\mu}$ with Lie algebra $(\ggo,\underline{h}\cdot \mu)$ such that there is an isomorphism $G\rightarrow G_{\underline{h}\cdot\mu}$ with derivative $\underline{h}$, one obtains an isometry between the following Riemannian homogeneous spaces,
\begin{equation}\label{isom}
(G/K,\la h\cdot,h\cdot\ra) \longrightarrow (G_{\underline{h}\cdot\mu}/K_{\underline{h}\cdot\mu},\ip),
\end{equation}
where $K_{\underline{h}\cdot\mu}$ is the image of $K$ under the isomorphism.  Note that $K_{\underline{h}\cdot\mu}$ is a Lie subgroup of $G_{\underline{h}\cdot\mu}$ with Lie algebra $(\kg,\underline{h}\cdot\mu|_{\kg\times\kg})$ and that $\ggo=\kg\oplus\pg$ is a reductive decomposition for every homogeneous space $G_{\underline{h}\cdot\mu}/K_{\underline{h}\cdot\mu}$, $h\in\sym_+(\pg)^K$.  Therefore, by varying the Lie brackets as in the right of \eqref{isom}, one is covering the whole set $\mca^G$ (see \cite{sol-HS} and references therein for further information).  

We assume from now on in this section that $G$ is unimodular (see \cite[\S 2.2]{PRP} for the general case).  According to \eqref{Ric2}, for any $h\in\sym_+(\pg)^K$, the Ricci operator of $(G_{\underline{h}\cdot\mu}/K_{\underline{h}\cdot\mu},\ip)$ is given by
\begin{equation}\label{Ric}
\Ricci_{\underline{h}\cdot\mu} = \Mm_{h\cdot\mu_\pg} - \unm h^{-1}\kil_\mu h^{-1}.  
\end{equation}
Note that $h^{-1}\kil_\mu h^{-1}$ is the Killing form operator of the Lie algebra $(\ggo,\underline{h}\cdot\mu)$   and by \eqref{mm2},
\begin{equation}\label{mm1}
\la\Mm_{h\cdot\mu_\pg},A\ra := \unc\la\theta(A)(h\cdot\mu_\pg),h\cdot\mu_\pg\ra, \qquad\forall A\in\glg(\pg).  
\end{equation}
It follows from \eqref{isom} that the Ricci tensor and the Ricci operator of each metric $g_h:=\la h\cdot,h\cdot\ra\in\mca^G$ are respectively given by 
$$
\ricci(g_h)=\la h\Ricci_{\underline{h}\cdot\mu}h\cdot,\cdot\ra, \qquad \Ricci(g_h) = h^{-1}\Ricci_{\underline{h}\cdot\mu}h, \qquad\forall h\in\sym_+(\pg)^K,
$$
and by \eqref{scal2}, 
\begin{equation}\label{scal}
\scalar(g_h) = -\unc| h\cdot\mu_\pg|^2  - \unm\tr{\kil_\mu}h^{-2}.
\end{equation}
In order to study the different types of $G$-stability and $G$-deformability (see Definition \ref{stab-def}), using the moving-bracket approach described above, we consider the functions
\begin{equation}\label{Rc-MBA}
\overline{\ricci}, \overline{\E}:\sym_+(\pg)^K\longrightarrow\sym^2(\pg)^K, \qquad \overline{\scalar}:\sym_+(\pg)^K\longrightarrow\RR, 
\end{equation}
defined by $\overline{\ricci}(h):=\ricci(g_h)$, $
\overline{\scalar}(h):=\scalar(g_h)$ and $\overline{\E}(h):=\E(g_h)=\overline{\ricci}(h)-\frac{\overline{\scalar}(h)}{n}g_h$, for any $h\in\sym_+(\pg)^K$.

\subsection{First variation of $\scalar$}\label{FVSc} 
Let $S:\glg(\pg)\rightarrow\sym(\pg)$ denote the symmetric part operator $S(A):=\unm(A+A^t)$ relative to $\ip$.   

\begin{lemma}\label{dSc}
At any $h\in\sym_+(\pg)^K$, if $h(t)\in\sym_+(\pg)^K$, $h(0)=h$, $h'(0)=A$ (e.g., $h(t)=h+tA$ or $h(t)=he^{th^{-1}A}$), then
$$
\overline{\scalar}'_h(A) :=\left.\ddt\right|_0 \overline{\scalar}(h(t)) = -2\la\Ricci_{\underline{h}\cdot\mu},S(Ah^{-1})\ra, \qquad\forall A\in\sym(\pg)^K.  
$$
\end{lemma}

\begin{remark}
At the background metric $g$, i.e., $h=I$, in accordance with \eqref{grad-sc}, the following simpler formula holds: 
$$
\scalar'_g(T):=\left.\ddt\right|_0 \scalar(g+tT) =\unm\overline{\scalar}'_I(A) = -\la\Ricci_{\mu},A\ra = -\la\ricci(g),T\ra_g,   
$$   
for any $A\in\sym(\pg)^K$, where $T\in\sca^2(M)^G$, $T_o=\la A\cdot,\cdot\ra\in \sym^2(\pg)^K$.
\end{remark}

\begin{proof} 
We first give the following useful formula, which is easy to prove using \eqref{tita}:
\begin{equation}\label{ddthmu}
\ddt \left(h(t)\cdot\mu_\pg\right) = \theta\left(h'(t)h(t)^{-1}\right)\left(h(t)\cdot\mu_\pg\right).  
\end{equation}
It now follows from \eqref{scal} and \eqref{ddthmu} that 
\begin{align*}
\left.\ddt\right|_0 \overline{\scalar}(h(t)) =& -\unc\left.\ddt\right|_0 |h(t)\cdot\mu_\pg|^2 - \unm\left.\ddt\right|_0 \tr{\kil_\mu h(t)^{-2}} \\ 
=& -\unm\la\left.\ddt\right|_0 h(t)\cdot\mu_\pg,h\cdot\mu_\pg\ra - \unm \tr{\kil_\mu \left.\ddt\right|_0h(t)^{-2}} \\ 
=& -\unm\la\theta(Ah^{-1})h\cdot\mu_\pg, h\cdot\mu_\pg\ra + \unm\tr{\kil_\mu h^{-1}Ah^{-2}} + \unm\tr{\kil_\mu h^{-2}Ah^{-1}} \\  
=& -2\la\Mm_{h\cdot\mu_\pg},Ah^{-1}\ra + \tr{h^{-1}\kil_{\mu}h^{-1}S(Ah^{-1})} \\ 
=& -2\la\Ricci_{\underline{h}\cdot\mu},S(Ah^{-1})\ra,
\end{align*}
where the last equality follows from \eqref{Ric}.  
\end{proof}

Since $d\det|_hA=(\det{h})\tr{Ah^{-1}}$, if 
$$
\sym_+(\pg)_1:=\{ h\in\sym_+(\pg):\det{h}=1\},
$$ 
then 
$$
T_h\sym_+(\pg)_1^K = \left\{ A\in\sym(\pg)^K : \tr{Ah^{-1}}=0\right\},
$$
so the following corollary analogous to Lemma \ref{Palais} follows.  

\begin{corollary}\label{dSc-cor}
$h\in\sym_+(\pg)_1^K$ is a critical point of $\overline{\scalar}:\sym_+(\pg)_1^K\longrightarrow\RR$ if and only if the metric $g_h\in\mca^G$ is Einstein.    
\end{corollary}

\subsection{First variation of $\ricci$}\label{FVRc}
The derivative of the Ricci curvature function at the background metric $g\in\mca^G$ ($\ip=g_o$) was computed in \cite{PRP}.  We consider the maps
$$
\delta_{\mu_{\pg}}:\glg(\pg)\longrightarrow\Lambda^2\pg^*\otimes\pg, \qquad \delta_{\mu_{\pg}}^t:\Lambda^2\pg^*\otimes\pg\longrightarrow\glg(\pg),
$$
where $\delta_{\mu_{\pg}}(A):=-\theta(A)\mu_\pg$ (see \eqref{tita}) and $\delta_{\mu_{\pg}}^t$ is the transpose of $\delta_{\mu_{\pg}}$, and define the following operator,
\begin{equation}\label{Cmup}
\lic_\pg=\lic_\pg(g):\sym(\pg)\longrightarrow\sym(\pg), \qquad \lic_\pg A:=\unm S\circ\delta_{\mu_{\pg}}^t\delta_{\mu_{\pg}}(A)+A\Mm_{\mu_\pg} + \Mm_{\mu_\pg}A.  
\end{equation}

By using \eqref{mm2}, it is easy to check that $\lic_\pg$ satisfies the following properties (see \cite{PRP}):  
 
\begin{enumerate}[{\small $\bullet$}] 
\item $\lic_\pg$ is a self-adjoint operator.  

\item $\lic_\pg I=0$ since $\delta_{\mu_{\pg}}(I)=\mu_{\pg}$ and $
\delta_{\mu_{\pg}}^t\delta_{\mu_{\pg}}(I)= \delta_{\mu_{\pg}}^t(\mu_{\pg})=-4\Mm_{\mu_{\pg}}$.  Thus $\lic_\pg\sym(\pg)\subset\sym_0(\pg):=\{ A\in\sym(\pg):\tr{A}=0\}$ by self-adjointness.  

\item $\la \lic_\pg A,A\ra = \unm |\theta(A)\mu_\pg|^2 + 2\tr{\Mm_{\mu_\pg}A^2} = \unm \la\left(\theta(A)^2+\theta(A^2)\right)\mu_\pg,\mu_\pg\ra$, for any $A\in\sym(\pg)$.  

\item $\lic_\pg\sym(\pg)^K\subset\sym(\pg)^K$.  This follows by a straightforward computation using that $\Ad(z)\in\Aut(\ggo,\mu)$ and $\Ad(z)|_\pg$ is $\ip$-orthogonal for any $z\in K$.  

\item Moreover, $\lic_\pg\sym(\pg)^H\subset\sym(\pg)^H$ for any $g\in\mca^H$, where $H$ is any intermediate subgroup $K\subset H\subset N_G(K)$.  
\end{enumerate}

\begin{lemma}\label{dRicmu}\cite[Lemma 6.1]{PRP}
For any $T\in\sca^2(M)^G$, $T_o=\la A\cdot,\cdot\ra$, $A\in\sym(\pg)^K$, 
$$
d\ricci|_gT =  \unm d\overline{\ricci}|_I A=\unm \la\lic_\pg A\cdot,\cdot\ra.
$$ 
\end{lemma}

Since $\Delta_LT = 2d\ricci|_gT$ on $\tca\tca_g$ (see \cite[12.28']{Bss}), the following formula follows.  

\begin{corollary}\label{LL}\cite[Corollary 6.7]{PRP}
Let $M=G/K$ be a homogeneous space with $G$ compact, endowed with a reductive decomposition $\ggo=\kg\oplus\pg$.  Then the Lichnerowicz Laplacian $\Delta_L$ of any $G$-invariant Riemannian metric $g$ on $M$ is given by
$$
\Delta_LT = \la\lic_\pg A\cdot,\cdot\ra, \qquad\forall T\in\tca\tca_g^G,
$$
where $T_o=\la A\cdot,\cdot\ra\in \sym^2(\pg)^K \equiv \sca^2(M)^G$, $A\in\sym(\pg)^K$ and $\ip = g_o$.
\end{corollary}

Recall from \S \ref{G-TT-sec} the computation of the space $\Ker\delta_g\cap\sca^2(M)^G$ of $G$-invariant divergence-free symmetric $2$-tensors.

\subsection{Second variation of $\scalar$}\label{SVSc}

As expected, at an Einstein metric, the second derivative of the scalar curvature is strongly related to the first derivative of the Ricci curvature.  

\begin{lemma}\label{ddSc}
Suppose that the background metric $g$ is Einstein, say $\ricci(g)=\rho g$.  Then, for any $T\in\sca^2(M)^G$, $T_o=\la A\cdot,\cdot\ra$, $A\in\sym(\pg)^K$,   
\begin{align*}
\scalar''_g(T,T) =& \unc\overline{\scalar}''_I(A,A) := \unc\left.\dsdt\right|_0 \overline{\scalar}(h(t)) \\ 
=& -\unm\la \lic_\pg A,A\ra + \rho\tr{A^2} = \unm\la (2\rho\id-\lic_\pg)A,A\ra, 
\end{align*}
where $h(t)\in\sym_+(\pg)^K$, $h(0)=I$, $h'(0)=A$.   
\end{lemma}

\begin{remark}
Alternatively, $\scalar''_g(T,T) = -\unc|\theta(A)\mu_\pg|^2 - \unm\tr{\kil_\mu A^2}$, which follows from the fact that $M_{\mu_\pg}-\unm\kil_\mu=\rho I$.  
\end{remark}

\begin{remark}
Since $I$ is a critical point of $\overline{\scalar}|_{\sym_1(\pg)^K}$, the value of $\overline{\scalar}''_I(A,A)$ is well defined if $\tr{A}=0$, in the sense that it can be computed using any curve $h(t)\in\sym_1(\pg)^K$ through $I$ with velocity $A$.  On the other hand, $\overline{\scalar}''_I(I,I)=4\scalar(I)$, so the formula also holds for $A=I$ and thus $\scalar''_I(A,A)$ is well defined for any $A$.    
\end{remark}

\begin{proof} 
If $h(t):=e^{tA}$, then $\ddt h(t)\cdot\mu_\pg = \theta(A)(h(t)\cdot\mu_\pg)$ by \eqref{ddthmu}, and so
\begin{align*}
\left.\dsdt\right|_0 \overline{\scalar}(h(t)) =& -\unm\left.\ddt\right|_0 \la\theta(A)h(t)\cdot\mu_\pg,h(t)\cdot\mu_\pg\ra + \left.\ddt\right|_0 \tr{Ae^{-2tA}\kil_\mu} \\ 
=& -\unm \la\theta(A)^2\mu_\pg,\mu_\pg\ra -\unm \la\theta(A)\mu_\pg,\theta(A)\mu_\pg\ra -2 \tr{A^2\kil_\mu}\\ 
=& -\la\delta_{\mu_\pg}^t\delta_{\mu_\pg}(A),A\ra - 2 \tr{\kil_\mu A^2}\\ 
=& -2\la \lic_\pg A,A\ra + 4\tr{\Mm_{\mu_\pg}A^2} - 2 \tr{\kil_\mu A^2}\\
=& -2\la \lic_\pg A,A\ra + 4\tr{\Ricci_\mu A^2}.
\end{align*}
We are using formula \eqref{Cmup} in the second last equality.  The fact that $\Ricci_\mu=\rho I$ concludes the proof.  
\end{proof}

\subsection{First variation of $\E$}\label{FV-E}
The following formula for the derivative of the Einstein operator follows from Lemma \ref{dRicmu}.  
 
\begin{lemma}\label{dE}
If $g$ is Einstein, say $\ricci(g) = \rho g$, then
$$
d\E|_gT = \unm d\overline{\E}|_IA = \left\la\left(\unm \lic_\pg A - \rho A\right)\cdot,\cdot\right\ra + \tfrac{1}{n}\rho(\tr{A})\ip,
$$
for any $T\in\sca^2(M)^G$, $T_o=\la A\cdot,\cdot\ra$, $A\in\sym(\pg)^K$.  
\end{lemma}

In particular, $d\E|_g=d\ricci|_g-\rho\id$ restricted to $(\RR g)^{\perp_g}$.

\subsection{Stability in terms of $\lic_\pg$}\label{stabCmup-sec}
We assume in this subsection that the background metric $g\in\mca^G$ is Einstein.  Under the identifications in terms of operators, the decomposition of the space of variations analogous to \eqref{tg2} is the following decomposition of the tangent space $T_I\sym_+(\pg)^K=\sym(\pg)^K$ at the identity map $I$:
\begin{equation}\label{tg3}
T_I\sym_+(\pg)^K = \left(\RR I \oplus T_I\Aut(G/K)\cdot I\right) \oplus^\perp W, 
\end{equation}
where $W$ is the $\ip$-orthogonal complement of $\RR I \oplus\Aut(G/K)\cdot I$ and $\Aut(G/K)$ acts on $\sym_+(\pg)^K$ according to the identification $\sym_+(\pg)^K\equiv\mca^G$.  Recall that if $G$ is compact, then $\tca\tca_g^G=\la W\cdot,\cdot\ra$ by Corollary \ref{tan4}.  Note that 
$$
W\subset\sym_0(\pg)^K, 
$$ 
and if in addition any of the conditions listed at the end of \S \ref{G-trivial} holds, then $W=\sym_0(\pg)^K$.  

It follows from \cite[Lemma 6.10]{PRP} that $d\overline{\ricci}|_IS(D)=2\rho S(D)$ for any $\underline{D}\in\Der(\ggo/\kg)$.  We therefore obtain from \eqref{tg} and Lemma \ref{dRicmu} that 
\begin{equation}\label{Cder}
\lic_\pg|_{T_I\Aut(G/K)\cdot I}=2\rho\id,  
\end{equation}
where $\lic_\pg$ is the operator attached to the metric $g$ as in \eqref{Cmup}.  

According to Definition \ref{stab-def}, it follows from Lemmas \ref{ddSc} and \ref{dE} that the $G$-stability and $G$-deformability types of the Einstein metric $g$ are both determined by the spectrum of the operator $\lic_\pg$ restricted to $W$, which coincides with the Lichnerowicz Laplacian in the compact case (see Corollary \ref{LL}).  All this is summarized in the following proposition.  

Let $\lambda_\pg=\lambda_{\pg}(g)$  and $\lambda_\pg^{max}=\lambda_{\pg}^{max}(g)$ denote, respectively,  the minimum and maximum eigenvalue of  $\lic_\pg=\lic_\pg(g)$ restricted to the subspace $W$ defined in \eqref{tg3}.  

\begin{proposition}\label{stab-Cmup}
Let $g$ be a $G$-invariant metric on a homogeneous space $M=G/K$, where $G$ is unimodular, endowed with a reductive decomposition $\ggo=\kg\oplus\pg$.  If $g$ is Einstein, say $\ricci(g)=\rho g$, then the following holds: 

\begin{enumerate}[{\rm (i)}] 
\item $g$ is $G$-stable if and only if $2\rho<\lambda_\pg$.

\item $g$ is $G$-unstable if and only if $\lambda_\pg<2\rho$.

\item $g$ is $G$-non-degenerate if and only if $G$-infinitesimally non-deformable, if and only if $2\rho\notin\Spec\left(\lic_\pg |_{W}\right)$.

\item $g$ is $G$-neutrally stable if and only if $\lambda_\pg=2\rho$.  

\item $g$ is $G$-strongly unstable if and only if $\lambda_\pg^{max}<2\rho$.
\end{enumerate}
\end{proposition}

\begin{remark}\label{Cmup-rem} 
It follows from Corollary \ref{LL} that $\lambda_L(g) \leq \lambda_\pg(g)$ (see \S \ref{stab-sec}).    
\end{remark}

In the case of a product homogeneous space, i.e., $G=G_1\times G_2$, $K=K_1\times K_2$, $K_i\subset G_i$, $\ggo_i=\kg_i\oplus\pg_i$ and $g=g_1+g_2$, where $g_i$ is a $G_i$-invariant metric on $M_i=G_i/K_i$, we obtain that $W=W_1\oplus W_2\oplus \RR A_0$, where 
$$
A_0:=\left(n_2I_{\pg_1},-n_1I_{\pg_2}\right), \qquad n_i:=\dim{M_i},
$$
and $\lic_\pg(g)=\lic_{\pg_1}(g_1)+\lic_{\pg_2}(g_2)$.   Since $A_0\in W\cap\Ker\lic_\pg$, one deduces that $\lambda_\pg\leq 0$ and so any positive scalar curvature homogeneous product Einstein metric $g$ is $G$-unstable.

\section{Naturally reductive case}\label{natred-sec}

We consider in this section the case when $g\in\mca^G$ is a {\it naturally reductive} metric on $M$ with respect to $G$ and some reductive decomposition $\ggo=\kg\oplus\pg$, i.e., the map $\ad_\pg{X}:\pg\rightarrow\pg$ is skew-symmetric for any $X\in\pg$ (see \eqref{adp} or \eqref{mup}).  Note that $G$ is necessarily unimodular.  We refer to \cite[\S 7]{PRP} and references therein for further information on naturally reductive metrics.  

The moment map takes the simpler form 
\begin{equation}\label{mm-nr} 
\Mm_{\mu_\pg}=\unc\sum (\ad_\pg{X_i})^2, 
\end{equation} 
and the operator $\lic_\pg$ also considerably simplifies in the naturally reductive setting (see Lemma \ref{dRicmu} and \cite[Lemma 7.18]{PRP}):  
\begin{equation}\label{Cp-def}
\lic_\pg A:=-\unm\sum[\ad_\pg{X_i},[\ad_\pg{X_i},A]], \qquad\forall A\in\sym(\pg)^K,
\end{equation}
where $\{ X_i\}$ is any orthonormal basis of $(\pg,\ip)$ and $\ip=g_o$.  Note that $\lic_\pg\geq 0$; in particular, $\lambda_\pg\geq 0$.  We also recall that $T_g\mca^G=\RR g\oplus\tca\tca_g^G$ in the compact case, i.e., $W=\sym_0(\pg)^K$.  

Since $\lic_\pg A=0$ if and only if $[A,\ad_\pg{\pg}]=0$, the following conditions are equivalent by results due to Kostant \cite{Kst2} (see \cite[\S 7.1]{PRP}):

\begin{enumerate}[{\small $\bullet$}] 
\item $\Ker\lic_\pg=\RR I$. 

\item $g$ is, up to scaling, the unique naturally reductive metric on $M$ with respect to $G$ and $\pg$.

\item $g$ is holonomy irreducible.   

\item $(\widetilde{M},g)$ is de Rham irreducible, where $\widetilde{M}$ denotes the simply connected cover of $M$.  

\item $\kg$ is $\ggo$-indecomposable, in the sense that there exist no nonzero ideals $\ggo_1$ and $\ggo_2$ of $\ggo$ such that $\ggo=\ggo_1\oplus\ggo_2$ and $\kg=\kg\cap\ggo_1\oplus\kg\cap\ggo_2$ (e.g., if $\ggo$ is indecomposable).  
\end{enumerate}

\subsection{Killing metrics on Lie groups}
For $M=G$ a compact semisimple Lie group, we consider the left-invariant metric $g_{\kil}\in\mca^G$ defined by $-\kil_{\ggo}$, where $\kil_{\ggo}$ denotes the Killing form of $\ggo$.  According to \eqref{mm-nr}, $M_{\mu_\pg}=-\unc\cas_{\ad,-\kil_\ggo}=-\unc I$, the Casimir operator acting on the adjoint representation of $\ggo$, and so $\ricci(g_{\kil})=\unc g_{\kil}$ by \eqref{Ric2}.  On the other hand, $\pg=\ggo$ and it follows from \eqref{Cp-def} that 
$$
\lic_\pg=\unm \cas_\tau, 
$$
where $\cas_\tau=\cas_{\tau,-\kil_\ggo}$ is the Casimir operator acting on the representation $\sym(\ggo)$ of $\ggo$ given by 
$$
\tau(X)A:=[\ad{X},A],
$$ 
i.e., $\cas_\tau=-\sum\tau(X_i)^2$, where $\{ X_i\}$ is a $-\kil_\ggo$-orthonormal basis of $\ggo$.  The first positive eigenvalue $\lambda_\tau$ of $\cas_\tau$ can therefore be computed by using representation theory.  We have collected in Table \ref{table1} the values of $\lambda_\tau$ for each simple Lie algebra $\ggo$, which together with Proposition \ref{stab-Cmup}, give the following.  Note that $\lambda_\pg=\unm\lambda_\tau$ and $2\rho=\unm$.  

\begin{table}
$$
\begin{array}{c|c|c|c|c}
\text{Type} & \ggo & n& \lambda_\tau & \text{Stab. type}  
\\[2mm] \hline \hline \rule{0pt}{14pt}
\textup{A}_1 &\sug(2) & & 3 & G\text{-stable} 
\\[2mm]  \hline \rule{0pt}{14pt}
\textup{A}_n &\sug(n+1) &n\geq 2 &  1  & G\text{-neut. stab.} 
\\ [2mm] \hline\rule{0pt}{14pt}
\textup{B}_3 &\sog(7) &  & \frac{6}{5}  & G\text{-stable} 
\\[2mm] \hline\rule{0pt}{14pt}
\textup{B}_n &\sog(2n+1) &  n\geq 4& \frac{2n+1}{2n-1}  & G\text{-stable}  
\\[2mm] \hline\rule{0pt}{14pt}
\textup{C}_n &\spg(n) & n\geq 2& \frac{n}{n+1}  & G\text{-unstable} 
\\[2mm] \hline\rule{0pt}{14pt}
\textup{D}_n& \sog(2n) &  n\geq 4 & \frac{n}{n-1} & G\text{-stable}  
\\[2mm] \hline\rule{0pt}{14pt}
\textup{E}_{6} & \eg_6 & & \frac{3}{2} & G\text{-stable}  
\\[2mm] \hline\rule{0pt}{14pt}
\textup{E}_{7} & \eg_7 && \frac{14}{9} & G\text{-stable}  
\\[2mm] \hline\rule{0pt}{14pt}
\textup{E}_{8} &\eg_8 && \frac{8}{5} & G\text{-stable} 
\\[2mm] \hline\rule{0pt}{14pt}
\textup{F}_{4} &\fg_4&& \frac{13}{9} & G\text{-stable}  
\\[2mm] \hline\rule{0pt}{14pt}
\textup{G}_{2} &\ggo_2 && \frac{7}{6} & G\text{-stable} 
\\[2mm] \hline\hline
\end{array}
$$
\caption{First eigenvalue $\lambda_\tau$ of the Casimir operator $\cas_\tau$ acting on $\sym(\ggo)$ with respect to $-\kil_\ggo$ for a compact simple $\ggo$} \label{table1}
\end{table}

\begin{proposition}\label{bi-inv}
Let $G$ be a connected compact simple Lie group and let $g_{\kil}$ denote the Killing metric, which is Einstein with $\ricci(g_{\kil})=\unc g_{\kil}$.     

\begin{enumerate}[{\small $\bullet$}] 
\item For $G=\SU(n)$, $n\geq 3$, the metric $g_{\kil}$ is $G$-neutrally stable with nullity $n^2-1$.   

\item $g_{\kil}$ is $G$-unstable on any $G=\Spe(n)$, $n\geq 2$, with coindex $\geq \frac{2n(2n-1)}{2}-1$. 

\item In all the remaining cases, $g_{\kil}$ is $G$-stable. 
\end{enumerate}
\end{proposition}

In particular, $g_{\kil}$ is a local maximum of $\scalar|_{\mca_1^G}$ in most of the cases.  The question of whether $g_{\kil}$ on $\SU(n)$ is a local maximum of $\scalar|_{\mca_1^G}$ or not is still open for $n\geq 4$.  It was proved in \cite{Jns} that it is not for $n=3$, while it is well known that it is a global maximum for $n=2$.  Concerning $\Spe(n)$, since $\lambda_\pg^{max}= \frac{2n+4}{2(n+1)}>\unm=2\rho$, $g_{\kil}$ is a saddle point of $\scalar|_{\mca_1^G}$.

This shows that the picture in the $G$-invariant setting is completely analogous to the general case studied by Koiso in \cite{Kso}, as described at the end of \S \ref{stab-sec}.  In particular, any bi-invariant metric on any compact simple Lie group $G$ is $G$-rigid, except possibly for $\SU(n)$, $n\geq 3$.  Nevertheless, it was proved in \cite[Theorem 22.3]{DrdGal} that on $\SU(n)$, $g_{\kil}$ is indeed $G$-rigid.

\subsection{A formula for $\lic_\pg$ in terms of structural constants}\label{strconst-sec}  
Let $M=G/K$ be a homogeneous space with $G$ compact and reductive decomposition $\ggo=\kg\oplus\pg$.   Given a non-degenerate $\ad{\ggo}$-invariant symmetric bilinear form $Q$ on $\ggo$ such that $Q(\kg,\pg)=0$ and $Q|_{\pg}>0$, we consider the metric $g_Q\in\mca^G$ whose value at $o$ is $Q|_{\pg}$.  Thus $g_Q$ is naturally reductive with respect to $G$ and $\pg$.  

\begin{remark}
According to \cite[Theorem 4]{Kst2} (see also \cite[p.4]{DtrZll}), if $\ggo=\pg+[\pg,\pg]$, then any $G$-invariant metric on $M$ which is naturally reductive with respect to $G$ and $\pg$ is given in this way for a unique $Q$.  
\end{remark}

Recall that $g$ is called {\it normal} when $Q>0$, and if in addition $G$ is semi-simple and $Q=-\kil_{\ggo}$, then $g$ is called {\it standard}.  In particular, if $G$ is simple, then $g_Q$ is necessarily standard (up to scaling).  

Given any $Q$-orthogonal decomposition $\pg=\pg_1\oplus\dots\oplus\pg_r$ in $\Ad(K)$-invariant and irreducible subspaces $\pg_1,\dots,\pg_r$ ($d_i:=\dim{\pg_i}$), we consider the corresponding structural constants given by,
$$
[ijk]:=\sum_{\alpha,\beta,\gamma} Q([X_\alpha^i,X_\beta^j], X_\gamma^k)^2,
$$
where $\{ X_\alpha^i\}$ is a $Q$-orthonormal basis of $\pg_i$.  Since $g_Q$ is naturally reductive relative to $G$ and $\pg$, the number $[ijk]$ is invariant under any permutation of $ijk$.  

Recall from \eqref{Ric2} that the Ricci operator of $g$ is given by $\Ricci(g_Q)=\Mm_{\mu_\pg}-\unm\kil_\mu$.  Since $\Mm_{\mu_\pg}$ is $\Ad(K)$-invariant, for each $k$ we have that the linear map $\Mm_{\mu_\pg}$ restricted to $\pg_k$ and composed with the orthogonal projection on $\pg_k$ is given by $m_kI_{\pg_k}$ for some $m_k\in\RR$.  It follows from \eqref{mm4} that
\begin{equation}\label{mkstr}
m_k = -\tfrac{1}{4d_k}\sum_{i,j} [ijk], \qquad\forall k=1,\dots,r;  
\end{equation}
indeed, 
\begin{align*}
m_kd_k =& \tr{\Mm_{\mu_\pg}|_{\pg_k}} = \sum Q( \Mm_{\mu_\pg}X_\gamma^k,X_\gamma^k)\\  
=& -\unm\sum Q(\mu_\pg(X_\gamma^k,X_\alpha^i),X_\beta^j)^2 +\unc\sum Q(\mu_\pg(X_\alpha^i,X_\beta^j),X_\gamma^k)^2 \\
=& -\tfrac{1}{2}\sum [kij] +\tfrac{1}{4}\sum [ijk] =-\tfrac{1}{4}\sum [ijk].  
\end{align*}
Note that this can alternatively be computed using \eqref{mm-nr}.  The irreducibility of $\pg_k$ also gives that $-\kil_\ggo$ restricted to $\pg_k$ equals $b_kQ|_{\pg_k}$ for some $b_k\in\RR$ for any $k$, and consequently, the restriction and projection of $\kil_\mu$ is given by $-b_kI_{\pg_k}$.  Note that $b_k\geq 0$, where equality holds if and only if $\pg_k\subset\zg(\ggo)$, and that if $g_Q$ is standard, then $b_k=1$ for all $k$.  We therefore obtain that 
\begin{equation}\label{rick-nr}
\Ricci(g_Q)|_{\pg_k}=\rho_kI_{\pg_k}, \qquad \rho_k = \tfrac{b_k}{2} -\tfrac{1}{4d_k}\sum_{i,j} [ijk],   
\end{equation}
and the Einstein equations become: $\Ricci(g_Q)=\rho I$ if and only if 
$$
\rho_k=\rho, \quad\forall k=1,\dots,r \quad\mbox{and}\quad \la\Ricci(g_Q)\pg_i,\pg_j\ra=0, \quad\forall i\ne j.  
$$
We now assume that the isotropy representation of the homogeneous space $M=G/K$ is multiplicity-free.  Thus the right-hand side Einstein conditions above automatically hold and 
$$
\left\{ \tfrac{1}{\sqrt{d_1}}I_{\pg_1},\dots, \tfrac{1}{\sqrt{d_r}}I_{\pg_r}\right\}
$$ 
is an orthonormal basis of $\sym(\pg)^K$.   Let $[\lic_\pg]$ denote the matrix of $\lic_\pg(g_Q)$ with respect to this basis.   

\begin{theorem}\label{Lpjk}
Let $g_Q\in\mca^G$ be the naturally reductive metric on $M=G/K$ ($G$ compact) attached to a non-degenerate $\ad{\ggo}$-invariant symmetric bilinear form $Q$ on $\ggo$, and assume that $G/K$ is multiplicity-free.  Then, the entries of the matrix $[\lic_\pg]$ are given by, 
$$
[\lic_\pg]_{kk} = 
\tfrac{1}{d_k}\sum_{\substack{j\ne k \\ i}} [ijk], \quad\forall k, 
\qquad 
[\lic_\pg]_{jk} =  
-\tfrac{1}{\sqrt{d_j}\sqrt{d_k}}\sum_{i} [ijk], \quad\forall j\ne k.  
$$
\end{theorem}

\begin{remark}
It is easy to check that the coordinates vector $[\sqrt{d_1},\dots,\sqrt{d_r}]^t$ of the identity map is indeed in the kernel of $[\lic_\pg]$.  Note that the structural constants of the form $[kkk]$ are not involved in the above formulas.  
\end{remark}

\begin{proof}
We fix any $Q$-orthonormal basis $\{ X_\alpha^i\}$ of each $\pg_i$ and denote  
$$
\ad_{\pg}{X_\alpha^i} = \left[\begin{matrix} 
\ad_{\pg_1}{X_\alpha^i}&(\ad_{\pg}{X_\alpha^i})_{12}&\cdots&(\ad_{\pg}{X_\alpha^i})_{1r}\\ 
-(\ad_{\pg}{X_\alpha^i})_{12}^t&\ad_{\pg_2}{X_\alpha^i}&\cdots&(\ad_{\pg}{X_\alpha^i})_{2r}\\ 
\vdots&\vdots&\ddots&\vdots\\ 
-(\ad_{\pg}{X_\alpha^i})_{1r}^t&-(\ad_{\pg}{X_\alpha^i})_{2r}^t&\cdots&\ad_{\pg_r}{X_\alpha^i} 
\end{matrix}\right], 
$$
where $(\ad_{\pg}{X_\alpha^i})_{jk}:\pg_k\rightarrow\pg_j$.  We also consider $E_{jk}:\pg_j\rightarrow\pg_j$ and $F_{jk}:\pg_k\rightarrow\pg_k$ defined by 
$$
E_{jk}:=\sum_{i,\alpha}(\ad_{\pg}{X_\alpha^i})_{jk}(\ad_{\pg}{X_\alpha^i})_{jk}^t,  \qquad
F_{jk}:=\sum_{i,\alpha}(\ad_{\pg}{X_\alpha^i})_{jk}^t(\ad_{\pg}{X_\alpha^i})_{jk}, \qquad \forall j<k.  
$$
For any diagonal block map 
$$
A:=\left[a_1I_{\pg_1}, a_2I_{\pg_2}, \dots, a_rI_{\pg_r} \right]\in\sym(\pg)^K,  
$$
a straightforward computation using \eqref{Cp-def} gives that the $k$-th block of $\lic_\pg A$ is given by
$$
\sum_{j<k}(a_k-a_j)F_{jk}+\sum_{k<j}(a_k-a_j)E_{kj}.  
$$
In particular, for each $l$, 
$$
\lic_\pg I_{\pg_l} = \left[-E_{1l},\dots,-E_{l-1,l}, \sum_{j=1}^{l-1}F_{jl}+ \sum_{j=l+1}^r E_{lj},-F_{l,l+1},\dots,-F_{lr}     \right]^t.  
$$
Since $\lic_\pg\sym(\pg)^K\subset\sym(\pg)^K$, this implies that $E_{jk}=e_{jk}I_{\pg_j}$ and $F_{jk}=f_{jk}I_{\pg_k}$ for all $j<k$, for some non-negative $e_{kj},f_{jk}\in\RR$.  But $\tr{E_{jk}}=\tr{F_{jk}}$, so 
$$
d_kf_{jk} = \tr{F_{jk}} = \sum_{i} [ijk], \qquad e_{jk}=\tfrac{d_k}{d_j}f_{jk},  \qquad \forall j<k,   
$$   
concluding the proof.  
\end{proof}

\section{Three standard infinite families}\label{unif-sec}

In this section, we assume that $M=G/K$ is one of the following: 
\begin{equation}\label{unif4}
\begin{array}{c}
\SU(nk)/\Se(\U(k)\times\dots\times\U(k)), \quad k\geq 1, \qquad \Spe(nk)/\Spe(k)\times\dots\times\Spe(k), \quad k\geq 1; \\ \\ 
\SO(nk)/\Se(\Or(k)\times\dots\times\Or(k)), \quad k\geq 3,  
\end{array}
\end{equation}
where the quotients are all $n$-times products with $n\geq 3$.  The standard block matrix reductive decomposition is given by 
$$
\ggo=\kg\oplus\pg_{12}\oplus\pg_{13}\oplus\dots\oplus\pg_{(n-1)n},
$$
where every $\pg_{ij}=\pg_{ji}$ (note that always $i\ne j$) has dimension $d=2k^2,4k^2,k^2$, respectively, and they are all $\Ad(K)$-irreducible and pairwise inequivalent.  Thus $G/K$ is multiplicity-free and $\dim{\mca^G}=\frac{n(n-1)}{2}$.  

It is easy to check that $[\pg_{ij},\pg_{kl}]_\pg=0$ if $\{ i,j\}$ and $\{ k,l\}$ are either equal or disjoint, and $[\pg_{ij},\pg_{ik}]_\pg$ is nonzero and it is contained in $\pg_{jk}$ for all $j\ne k$.  Moreover, a straightforward computation gives that any nonzero structural constant $[ijk]$ as in \S\ref{strconst-sec} is equal to the same $c=c(G,k,n)$, where $\frac{c}{d}$ is respectively given by     
\begin{equation}\label{unif1}
\frac{c}{d} \quad = \quad \frac{1}{2n}, \quad \frac{k}{2(nk+1)}, \quad \frac{k}{2(nk-2)}.  
\end{equation} 
We consider the standard or Killing metric $g_{\kil}$ on $G/K$, i.e., $Q=-\kil_\ggo$ (see \S\ref{strconst-sec}).  It follows from \eqref{rick-nr} that $g_{\kil}$ is Einstein with 
\begin{equation}\label{unif2}
2\rho=1-\frac{c}{d}(n-2).  
\end{equation}  
On the other hand, according to Theorem \ref{Lpjk}, 
$$
[\lic_\pg]_{(ij)(ij)} = \frac{c}{d}2(n-2), \qquad [\lic_\pg]_{(ij)(ik)} = -\frac{c}{d}, \quad\forall j\ne k, 
$$
and $[\lic_\pg]_{(ij)(kl)}=0$ otherwise.  This implies that 
$$
[\lic_\pg] = \frac{c}{d}\Big(2(n-2)I-\Adj(X)\Big), 
$$  
where $X=J(n,2,1)$ is the Johnson graph with parameters $(n,2,1)$ (see \cite[\S1.6]{GdsRyl}) and $\Adj(X)$ denotes its adjacency matrix.  Since the graph is strongly regular with parameters $(\frac{n(n-1)}{2},2(n-2),n-2,4)$ for any $n\geq 4$ (see \cite[\S10.1]{GdsRyl}), it follows from \cite[\S10.2]{GdsRyl} that the spectrum of $\Adj(X)$ is given by 
$$
2(n-2), \quad n-4, \quad -2,\qquad \mbox{with multiplicities}\qquad  1, \quad n-1, \quad \frac{n(n-3)}{2},  
$$    
respectively.  Thus $\Spec(\lic_\pg)=\{0,\lambda_\pg,\lambda_\pg^{max}\}$, where 
\begin{equation}\label{unif3}
\lambda_\pg = \frac{c}{d}n, \qquad \lambda_\pg^{max} = \frac{c}{d}2(n-1), \qquad n\geq 4,
\end{equation}  
and have multiplicities $n-1$ and $\frac{n(n-3)}{2}$, respectively.  

For $n=3$, $X$ is the complete graph on $3$ vertices and so the spectrum of $\Adj(X)$ equals $\{ 2, -1\}$, with multiplicities $1$ and $2$, respectively.  Thus $\lambda_\pg=\lambda_\pg^{max}=  \frac{c}{d}3$ and has multiplicity $2$ if $n=3$.  

The following proposition follows from a straightforward comparison between \eqref{unif1}, \eqref{unif2} and \eqref{unif3}.  

\begin{proposition}\label{unif-prop}
The standard metric $g_{\kil}$ on each of the homogeneous spaces given in \eqref{unif4} is always $G$-unstable, and so Ricci flow dynamically unstable.  The coindex and type of critical point are given in Table \ref{table2}.  They are all $G$-non-degenerate, and in particular $G$-rigid, except 
$$
\SU(4k)/\Se(\U(k)\times\U(k)\times\U(k)\times\U(k)), \quad k\geq 1, \qquad \Spe(10)/\Spe(2)^5, \qquad \Spe(6)/\Spe(1)^6.  
$$
\end{proposition}

We do not know whether $g_{\kil}$ is still a local minimum in the $G$-degenerate cases or not.  

\begin{table}
$$
\begin{array}{c|c|c|c|c}
G/K & n & k & \text{Crit.point} & \text{coindex}  
\\[2mm] \hline \hline \rule{0pt}{14pt}
\SU(3k)/\Se(\U(k)^3) & 3 & k\geq 1 & \text{loc.min.} & 2 
\\[2mm]  \hline \rule{0pt}{14pt}
\SU(4k)/\Se(\U(k)^4) & 4 & k\geq 1 & G\text{-deg.} & 3
\\ [2mm] \hline\rule{0pt}{14pt}
\SU(nk)/\Se(\U(k)^n) & n\geq 5 & k\geq 1 & \text{saddle}  & n-1 
\\[2mm] \hline\hline\rule{0pt}{14pt}
\Spe(3k)/\Spe(k)^3 & 3 & k\geq 1 & \text{loc.min.} & 2  
\\[2mm] \hline\rule{0pt}{14pt}
\Spe(4k)/\Spe(k)^4 & 4 & k\geq 1 & \text{loc.min.} & 5
\\[2mm] \hline\rule{0pt}{14pt}
\Spe(5)/\Spe(1)^5 & 5 & 1 & \text{loc.min.} & 9  
\\[2mm] \hline\rule{0pt}{14pt}
\Spe(10)/\Spe(2)^5 & 5 & 2 & G\text{-deg.} & 4  
\\[2mm] \hline\rule{0pt}{14pt}
\Spe(6)/\Spe(1)^6 & 6 & 1 & G\text{-deg.} & 5   
\\[2mm] \hline\rule{0pt}{14pt}
\Spe(kn)/\Spe(k)^n & n\geq 5 & \text{otherwise} & \text{saddle} & n-1 
\\[2mm] \hline\hline\rule{0pt}{14pt}
\SO(3k)/\Se(\Or(k)^3) & 3 & k\geq 3 & \text{loc.min.} & 2   
\\[2mm] \hline
\SO(nk)/\Se(\Or(k)^n) & n\geq 4 & k\geq 3 & \text{saddle} & n-1   
\\[2mm] \hline\hline
\end{array}
$$
\caption{Coindex and critical point type of the $G$-unstable Einstein metric $g_{\kil}$ on each of the spaces given in \eqref{unif4}.} \label{table2}
\end{table}

\section{Jensen's metrics}\label{DZ-sec}

Given a simple Lie group $H$ and a semisimple subgroup $K\subset H$, we consider the $\kil_\hg$-orthogonal decomposition $\hg = \ag\oplus\kg$ and the left-invariant metrics on $H$ defined by 
$$
g_t=-\kil_\hg|_{\ag} + t(-\kil_\hg)|_{\kg}, \qquad t>0.   
$$
Thus $g_1$ is the Killing metric on $H$.  On the other hand, it was proved in \cite{Zll} (see also \cite[Theorem 1]{DtrZll}) that for each $t\ne 1$, the metric $g_t$ is naturally reductive with respect to $G=H\times K$ (acting on $H$ by $(h,k)\cdot p:=hpk^{-1}$) and the reductive decomposition 
$$
\ggo=\Delta\kg\oplus\pg_t,  \qquad \pg_t:= \pg_{\ag}\oplus \pg_\kg, \qquad
\pg_{\ag}:=(\ag,0), \qquad \pg_\kg:=\{ (\tfrac{t}{1-t}Z,-Z):Z\in\kg\}.  
$$
Indeed, $g_t$ is identified with $g_{Q_t}$, where $Q_t$ is the non-degenerate $\ad{\ggo}$-invariant bilinear symmetric form on $\ggo=(\hg,0) \oplus (0,\kg)$ given by 
$$
Q_t:=-\kil_\hg +\tfrac{t}{1-t}(-\kil_\hg)|_{\kg},   
$$
since for any $Z\in\kg$, the $Q_t$-orthogonal projection of $(0,Z)$ on $\pg_t$ is $(t-1)(\tfrac{t}{1-t}Z,-Z)$.  Note that $g_t$ is normal (i.e., $Q_t>0$) if and only if $t<1$.  If $\kg=\kg_1 \oplus \dots \oplus \kg_r$ is a $\kil_\hg$-orthogonal decomposition in simple ideals of $\kg$, then 
\begin{equation}\label{pqdec}
\pg_t:= \pg_{\ag}\oplus \pg_1\oplus\dots\oplus\pg_r, \qquad \pg_i:=\{ (\tfrac{t}{1-t}Z,-Z):Z\in\kg_i\}, \quad i=1,\dots,r,
\end{equation}
is an $\Ad(\Delta K)$-invariant $Q_t$-orthogonal decomposition of $\pg_t$.  
 
We assume from now on that $\ag$ is $\Ad(K)$-irreducible (i.e., $H/K$ is isotropy irreducible) and that for some constant $c$, $\kil_{\kg_i}=c\kil_\hg|_{\kg_i}$ for any $i=1,\dots,r$.  In particular, the summands in \eqref{pqdec} are all $\Ad(\Delta K)$-irreducible and pairwise inequivalent, so $\dim{\mca^G_1}=r$.  It is easy to check that the only nonzero structural constants are $[jjj]$, $[j\ag\ag]$ and $[\ag\ag\ag]$ (see \S\ref{strconst-sec}), which are next computed.  

\begin{lemma}\label{ijkDZ}
For each $j=1,\dots,r$, 
$$
[jjj] =\tfrac{(2t-1)^2}{t}cd_j, \qquad [j\ag\ag] = t(1-c)d_j, \qquad [\ag\ag\ag] = d-2(1-c)k, 
$$
where $d_j:=\dim{\pg_j}=\dim{\kg_j}$, $d:=\dim{\pg_{\ag}}=\dim{\ag}$ and $k:=\dim{\kg}$.   
\end{lemma}

\begin{proof}
These are straightforward computations which use for $[jjj]$ that 
$$
I_{\kg_j}=\cas_{\ad,-\kil_{\kg_j}} =-\sum \left(\ad_{\kg_j}{\tfrac{1}{\sqrt{c}}Z_i^j}\right)^2, 
$$
where $\{ Z_i^j\}$ is any $-\kil_{\hg}$-orthonormal basis of $\kg_j$ (recall that $[jjj]=-\sum\limits_\alpha\tr(\ad{X_\alpha^j}|_{\pg_j})^2$ for any orthonormal basis $\{ X_\alpha^j\}_{\alpha=1}^{d_j}$ of $\pg_j$), and for $[j\ag\ag]$ and $[\ag\ag\ag]$ that 
$$
\sum a_j(X_i)^ta_j(X_i) = (1-c)I_{\kg_j}, \qquad\forall j=1,\dots,r,
$$
where 
$$
\ad_\hg{X_i} = \left[\begin{matrix} 
\ad_\ag{X_i} & a_0(X_i) &\cdots& a_r(X_i) \\ 
-a_0(X_i)^t & &&\\ 
\vdots & &0&\\
-a_r(X_i)^t & &&\\
\end{matrix}\right]  
$$
and $\{ X_i\}$ is a $-\kil_{\hg}$-orthonormal basis of $\ag$.  
\end{proof}

According to \cite[Corollary 2, p.44]{DtrZll}, if $t\ne 1$, then $\Ricci(g_t)=\rho I$ if and only if 
$$
t=t_E:=\frac{dc}{(d+2k)(1-c)}, \qquad 2\rho = \frac{c}{2t_E}+\frac{(1-c)t_E}{2}.
$$
We know from \cite[Theorem 11, (ii), p.35]{DtrZll} that 
$$
c<\frac{d+2k}{2d+2k},  \qquad\mbox{that is}, \qquad  t_E<1,
$$
as the exception $\spg(n-1)\subset\spg(n)$ does not appear in this case (see the last paragraph of the proof of \cite[Corollary 2, p.44]{DtrZll}).  In particular, $g_{t_E}$ is normal with respect to $G$ and $\pg_{t_E}$.  

It follows from Theorem \ref{Lpjk} and Lemma \ref{ijkDZ} that the matrix of the Lichnerowicz Laplacian $\lic_\pg(g_{t_E})$ relative to the orthonormal basis
$$
\left\{ \tfrac{1}{\sqrt{d}}I_{\pg_\ag}, \tfrac{1}{\sqrt{d_1}}I_{\pg_1},\dots,\tfrac{1}{\sqrt{d_r}}I_{\pg_r}\right\}, 
$$
of $\sym(\pg_{t_E})^{\Delta K}$ is given by 
$$
[\lic_\pg]=t_E(1-c)\left[\begin{matrix} 
\tfrac{k}{d}&-\frac{\sqrt{d_1}}{\sqrt{d}}&\cdots &-\frac{\sqrt{d_r}}{\sqrt{d}}\\ 
-\frac{\sqrt{d_1}}{\sqrt{d}}&1&\dots&0\\ 
\vdots&\vdots&\ddots&\vdots\\ 
-\frac{\sqrt{d_r}}{\sqrt{d}}&0&\cdots&1 
\end{matrix}\right].   
$$
Since the characteristic polynomial of $\frac{1}{t_E(1-c)}\lic_\pg$ is $f(x)=x(x-1)^{r-1}(x-(1+\frac{k}{d}))$, we obtain that 
$$
\Spec(\lic_\pg) = \left\{ 0, \; t_E(1-c), \;  t_E(1-c)(1+\tfrac{k}{d})\right\},  
$$
with multiplicities $1,r-1,1$, respectively, and so
$$
\left\{\begin{array}{ll} 
\lambda_\pg=t_E(1-c), \quad \lambda_\pg^{max}=t_E(1-c)(1+\tfrac{k}{d}) & \qquad r\geq 2, \\ \\
\lambda_\pg=\lambda_\pg^{max}=t_E(1-c)(1+\tfrac{k}{d}), & \qquad r=1.  
\end{array}\right.
$$

\begin{proposition}\label{J}
Every $g_{t_E}$ is $G$-unstable with coindex $r$ (in particular, $g_{t_E}$ is always a local minimum).  
\end{proposition}

\begin{proof} 
We have that
$$
\lambda_\pg^{max} = t_E(1-c)(1+\tfrac{k}{d}) < 2\rho = \frac{c}{2t_E}+\frac{(1-c)t_E}{2},
$$
if and only if
$$
0 < \frac{c}{2t_E} - t_E(1-c)(\unm+\tfrac{k}{d}) = \frac{c}{2t_E} - \frac{c}{2},
$$
if and only if $t_E<1$, as was to be shown.  
\end{proof}

If $\mca^H$ denotes the huge space of all left-invariant metrics on $H$, then $\mca^G$ is identified with the subset of $\mca^H$ of those metrics which are in addition $K$-invariant.  In particular, the Einstein metric $g_{t_E}$ is also $H$-unstable, that is, unstable as a left-invariant metric on $H$, and so Ricci flow dynamically unstable.  Recall that the $H$-stability type of the Killing metric $g_1$ on the Lie group $H$ has been established in Proposition \ref{bi-inv}.   

It follows from the lists of isotropy irreducible homogeneous spaces given in \cite[Tables 7.102, 7.106, 7.107]{Bss} that Proposition \ref{J} provides at least one $H$-unstable Einstein left-invariant metric on any simple Lie group, except $\Spe(2n+1)$, $n\geq 4$ and $\SO(n)$ for some odd $n$'s.  

The only cases $K\subset H$ with coindex $\geq 2$ (i.e., $K$ non-simple) are (see \cite[p.46]{DtrZll}):
\begin{align*}
&\SO(n)\times\SO(n)\subset\SO(2n), \qquad \Spe(n)\times\Spe(n)\subset\Spe(2n), \\ 
&\SU(n)\times\SU(n)\subset\SU(n^2) \quad (\text{tensor product}), \\ 
&\SU(3)\times\SU(3)\times\SU(3)\subset E_6 \qquad \Spe(3)\times G_2\subset E_7.  
\end{align*}

\end{document}